\documentclass[10pt,reqno]{amsart}

\usepackage{graphicx}
\usepackage{tabularx}
\graphicspath{ {./images/} }
\usepackage{xargs}
\usepackage{geometry}
\usepackage{fancyhdr} 
\usepackage{lastpage} 
\usepackage{extramarks} 
\usepackage[most]{tcolorbox} 
\usepackage{xcolor} 
\usepackage[hidelinks]{hyperref} 
\usepackage[permil]{overpic}
\usepackage{pict2e} 
\usepackage{listings}
\usepackage{amssymb}
\usepackage{amsthm}
\theoremstyle{plain}
\usepackage[
backend=biber,
style=numeric,
sorting=ynt
]{biblatex}
\usepackage{xargs}
\usepackage{float}
\usepackage{lipsum}
\usepackage{geometry}
\usepackage{fancyhdr} 
\usepackage{lastpage} 
\usepackage{extramarks} 
\usepackage[most]{tcolorbox} 

\usepackage{xcolor} 
\usepackage[hidelinks]{hyperref} 
\usepackage[permil]{overpic}
\usepackage{pict2e} 
\usepackage{listings}

\newtheorem*{theorem*}{Theorem}

\theoremstyle{notation}

\numberwithin{equation}{section}
\theoremstyle{plain}
\newtheorem{definition}{Definition}[section]
\newtheorem{theorem}[equation]{Theorem}
\newtheorem{Conjecture}[equation]{Conjecture}
\newtheorem{Question}[equation]{Question}

\newtheorem{remark}[equation]{Remark}
\newtheorem{corollary}[equation]{Corollary}
\newtheorem{lemma}[equation]{Lemma}

\newtheorem{proposition}[equation]{Proposition} 
\providecommand{\keywords}[1]
{
  \small	
  \textbf{\textit{Keywords---}} #1
}

\title{Using Erdős's methods to study Yorke's problems}
\author{Eran Igra and Valerii Sopin}
\address{\newline
1. Shanghai Institute for Mathematics and Interdisciplinary Sciences (SIMIS), Shanghai 200433, China\newline
\newline
2. Research Institute of Intelligent Complex Systems, Fudan University, Shanghai 
200433, China \newline}
\email{eranigra@simis.cn}
\email{VvS@myself.com, vsopin@simis.cn}
\begin{document}

\begin{abstract}
In this paper, we study the possible bifurcations of periodic orbits by analyzing them as graphs. In detail, we construct a collection of graphs which are an idealized versions of bifurcation diagrams and color them using the Mallet-Yorke Orbit Index (and the Lefschetz Fixed Point Theorem). The aforementioned allows to study the genericity of routes to chaos, as well as to gain insight into their possible complexity. In particular, our results can be interpreted as saying that there is no upper bound on the possible complexity of routes to chaos in high dimensional systems.
\end{abstract}

\maketitle
\section{Introduction}

Let $S$ be a closed, smooth, orientable manifold of dimension $d>1$ and let $f_t:S\to S$, $t\in[0,1]$, be an orientation-preserving smooth one-parameter family. Moreover, assume the dynamics of $f_0$ is "simple", while the dynamics of $f_1$ is "complex". For example, assume that a.e. initial condition on $S$ is attracted to a stable periodic orbit for $f_0$, while $f_1$ has a Smale Horseshoe dynamics, or alternatively, infinitely many sinks per the Newhouse Phenomenon (see \cite{New}). In this paper we ask the following question - can we tell which bifurcations we should expect for a generic choice of such a one-parameter family? In other words, which route to chaos involving periodic orbits should we expect? The motivation to study this question is clear. Ideally, such one-parameter families correspond to first-return maps to smooth flows - as such, understanding and being able to predict with a certain probability how their dynamics progress from order into chaos should give us tools to answer the same problem for one-parameter families of vector fields.

One natural approach to this question would be to try and define a measure on the Banach manifold all $C^\infty$ maps $F:S\times(0,1)\to S$, where $f_t(s)=F(s,t)$, $s\in S, t\in(0,1)$. That approach, however, has its limitations: as proven in \cite{MT}, if one were to construct a natural measure $m$ on the Banach Manifold $C^\infty(S\times(0,1),S)$ which is analogous to the Lebesgue measure, $m$ would only take the values $0$ and $\infty$. As proven in \cite{MT}, the measure $m$ is in many ways as best as can be, which implies it does not appear likely that one can assign probabilities to different types of routes to chaos. 

\begin{figure}[h]
\centering
\begin{overpic}[width=0.3\textwidth]{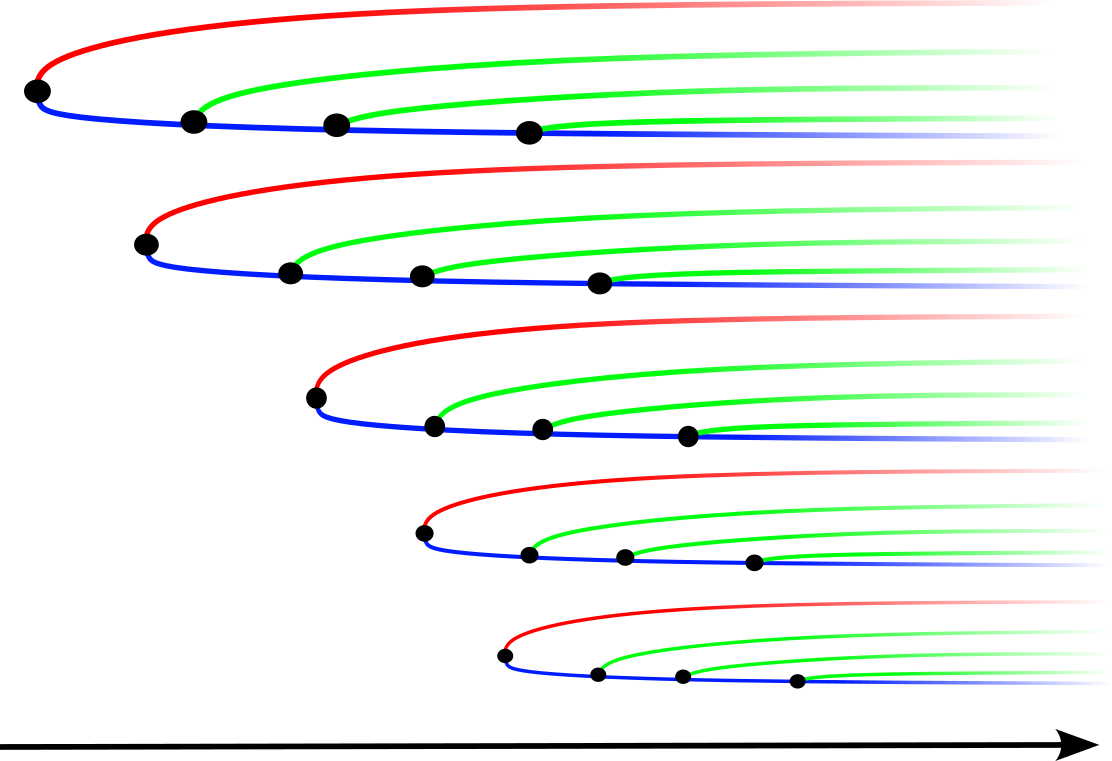}
\put(1030,0){$t$}
\end{overpic}
\caption{\textit{A (partial) illustration of a period-doubling route to chaos as a collection infinitely many trees, whose edges are colored w.r.t. the Mallet-Yorke Index \cite{KY4}, where $t$ denotes a parameter, while the dots denote bifurcation orbits.}}
\label{cascades}
\end{figure}

It is exactly this gap we study in this paper, where instead of a measurable approach we combine topological and graph theoretical approaches to tackle this problem. Inspired by the Mallet-Yorke Index originally introduced in \cite{PY} and developed in \cite{PY2}, \cite{PY3}, \cite{KE}, \cite{EY}, \cite{ky4}, \cite{KY4}, \cite{B1}, \cite{B2}, \cite{B3} and \cite{PY4} (among many others), we study graphs inspired by certain bifurcation laws, with which we attempt to give an explanation. In detail, we use the Mallet-Yorke Index to color a certain families of graphs, which give an inkling to the answer. The proposed answer correlates well with both theoretical and numerical studies.

To be precise, given any dimension $d>1$ and $k>0$ we construct $G^A_{k,d}$, a collection of colored acyclic graphs - i.e., colored trees - which are colored based on laws derived from the Mallet-Yorke Index. Intuitively, the colored graphs in $G^A_{k,d}$ should be interpreted as possible routes to chaos of periodic orbits for the one-parameter family $f_t:S\to S$, $t\in[0,1]$, under the restrictions that a given periodic orbit is allowed to split in a bifurcation into at most $k$ distinct orbits, and, that once a periodic orbit appears at some $t\in(0,1)$ it persists as $t\to1$ (see the discussion at the beginning of Section \ref{trees}). With that idea in mind, we consider the relative part of $G^A_{k,d}$ within $G^A_{k+1, d}$, which we denote by $\frac{|G^A_{k+1, d}|}{|G^A_{k,d}|}$ and treat as a direct limit. As we will prove in Theorem \ref{treeth}, for all $d>3,k>1$ we have $\frac{|G^A_{k+1,d}|}{|G^A_{k,d}|} = \frac{|G^A_{k+1,d}|}{|G^A_{2,d}|} = \frac{|G^A_{k,d}|}{|G^A_{1,d}|} = \infty$.

This result should be understood as follows: as the collection $\{G^A_{k,d}\}_{k}$ is interpreted as a collection of bifurcation diagrams describing routes to chaos of period multiplying bifurcations, then when the dimension of $S$ is $4$ and higher, there is no "most likely" route to chaos. In other words, our results suggest that in high dimension there are many more routes to chaos out there involving period multiplying bifurcations, and not a single one of them "dominates" the landscape. At this point, we remark that even though it is not at all clear whether the set $G^A_{k,d}$ truly represents all possible routes to chaos for the one-parameter family $f_t:S\to S$, our findings correlate well with several numerical observations and theoretical studies. For example, in \cite{Gallas} it was numerically observed that in general, one should not assume the classic period-doubling route to chaos to be the only generic case. Similarly, in \cite{Tur} it was shown how the introduction to mixed convection in a flow could alter the type of route to chaos. In a more theoretical direction, in \cite{GuW} it was concluded that there exist routes to chaos whose dynamics can be much more complex than the regular period-doubling route. Finally, as noted in \cite{AS}, as far as one can talk of a "likely" route to chaos, in higher-dimensional systems these routes to chaos appear to be based on Neimark-Sacker bifurcations, as opposed to the usual period-doubling one. 

Here, one may ask: do similar results also hold in lower dimensions, i.e., when the dimension of $S$ is either $2$ or $3$? Unfortunately, in such low dimensions the properties of the Mallet-Yorke Index alone, with which we define $G^A_{k,d}$, do not allow a clear answer. Nevertheless, as we will prove (drawing on the ideas of \cite{EY}), in these cases the size of $\frac{|G^A_{k+1, d}|}{|G^A_{k,d}|}$ strongly depends on how much $G^A_{k,d-1}$ fills up $G^A_{k,d}$ (see Corollary \ref{dim3} for the precise formulation). Specifically, this should be interpreted as follows:
\begin{itemize}
    \item When $d=3$, the answer depends on how many smooth one-parameter families of three-dimensional diffeomorphism have essentially two-dimensional behavior.
    \item When $d=2$, the answer depends on how many smooth one-parameter families of surface diffeomorphisms have "essentially" one-dimensional behavior.
\end{itemize}

This paper is organized as follows. In Section \ref{laws} we remind several basic facts about the Mallet-Yorke Index, use it to translate families of isotopies into colored graphs (see Propositions \ref{rules2}-\ref{rules4}) and rigorously define bifurcation diagrams (see Definition \ref{bifurcationgraph}). In Subsection \ref{trees} we study the relative size of $G^A_{k,d}$ in $G^A_{r, d}$, $k<r$, proving Theorem \ref{treeth} and discussing what happens in dimensions $2$ and $3$ in Corollary \ref{dim3}. 

\subsection*{Acknowledgments}
The authors would like to thank Asad Ullah for the helpful discussions. The first author would also like to thank Marian Gidea for introducing him to \cite{LM}, which inspired Corollary \ref{dim3}.

\section{Bifurcation Laws for Diffeomorphisms}
\label{laws}
In this Section we introduce the basic notations and definitions that will be used in the proof of Theorem \ref{treeth} in the next Section. This Section is organized as follows: we begin by recalling the terminology of the Mallet-Yorke Index as well as some of its basic facts, after which we study how the Mallet-Yorke Index correlates with the dimension of $S$ (Proposition \ref{rules2} - Proposition \ref{rules4}). Following that, we discuss the reduction of bifurcation diagrams to graphs, thus setting the stage for Theorem \ref{treeth}, which we will prove in the next Section.

To begin with, from now on unless otherwise stated, $S$ will always denote a closed orientable manifold of dimension $d>1$. Now, consider an orientation-preserving diffeomorphism $f:S\to S$. Let $x\in S$ be a periodic orbit of period $n$, where by "period" we always refer to the minimal period. Now, denote by $D_{f^n}(x)$ the differential of $f^n$ at some point of the periodic orbit $x$. Let $\sigma^+$ and $\sigma^-$ denote the number of eigenvalues of $D_{f^n}(x)$ at $(1,\infty)$ and $(-\infty,-1)$, respectively, and assume $\sigma^++\sigma^-=n$ (it is easy to see these quantities are independent of the choice of a point in the periodic orbit $x$). We define $\phi(x)$, the Mallet-Yorke Index of $x$, as follows:  $$
\phi(x)=\begin{cases}
			(-1)^{\sigma^+}, & \text{if $\sigma^-$ is even}\\
            0, & \text{if $\sigma^-$ is odd}
		 \end{cases}
$$
\begin{figure}[h]
\centering
\begin{overpic}[width=0.4\textwidth]{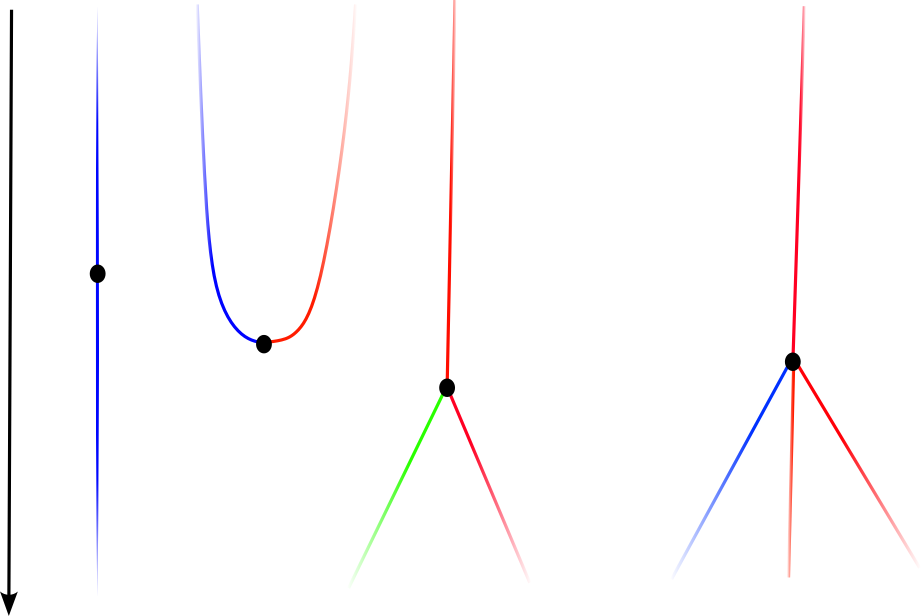}
\put(0,-40){$t$}
\end{overpic}
\caption{\textit{From left to right - a type $0$ orbit, type $1$, type $2$, and type $m\geq3$, represented by the black dot. The Mallet-Yorke Index is red when it is $-1$, blue when $1$ and green when $0$. }}
\label{fig1}
\end{figure}

This index can be used to study bifurcations. To state these results, we first define \textbf{type $m$} orbits for $m\geq1$. Let $x$ be a periodic orbit for $f_{t_0}$ of the minimal period $n$, where $f_t:S\to S$ is a $C^1$ isotopy of orientation-preserving diffeomorphisms, where $t\in[0,1]$, and $t_0$ is strictly interior to $[0,1]$ (from now on, we will always refer to such isotopies as \textbf{$C^1$ one-parameter families}). We now classify the types of periodic orbits we consider in this paper (see the illustrations in Figure \ref{fig1} and Figure \ref{decomposition}):

\begin{enumerate}
    \item $x$ is \textbf{type $0$} if $D_{f^n_{t_0}}(x)$ has no eigenvalues which are roots of unity. In this case, $x$ persists when $f_{t_0}$ is varied to $f_t$, $t\in(t_0-\epsilon,t_0+\epsilon)$, without changing its minimal period (for some $\epsilon>0$).
    \item $x$ is \textbf{type $1$} if $1$ is the only eigenvalue of $D_{f_0^n(x)}$ which is a root of unity (of multiplicity 1) and $x$ is the collision point for two periodic orbits $x',x''$ for $f_t$, which persist for $t\in(t_0-\epsilon,t_0)$, then collide at $x$ and disappear as $t$ is varied into $t>t_0$ (again, for some $\epsilon>0$).
    \item $x$ is \textbf{type $2$} if $-1$ is the only eigenvalue of $D_{f^n_{t_0}}(x)$ which is a root of unity (of multiplicity 1) and, when $t$ crosses into $(t_0,t_0+\epsilon)$, $x$ splits into two periodic orbits for $f_t$ - $x'$ and $x''$, s.t. $x'$ has period $n$ and $x''$ has period $2n$. Moreover, as we vary $t$ in the interval $(0,-\epsilon)$, $x$ persists as a periodic orbit for $f_t$ without changing its periodic orbit (again, for some $\epsilon>0$).
    \item  $x$ is \textbf{type $m$}, $m\geq3$, if the following is true:
    
    \begin{itemize}
        \item There exists some $k$, relatively prime to $m$, s.t. $D_{f_0^n(x)}$ has precisely two eigenvalues which are roots of unity, $e^{\pm 2\pi i\frac{k}{m}}$, each of multiplicity 1. 
        \item As we vary $f_{t_0}$ to $f_t$, when $t$ crosses into $t\in(t_0,t_0+\epsilon)$ the orbit $x$ splits into three orbits, $x',x''$ and $x'''$, s.t. $x''$ and $x'''$ have period $mn$. Moreover, the periodic orbit $x'$ retains the original period, i.e., n (again, for some $\epsilon>0$).
    \end{itemize}
    \item $x$ is a \textbf{$n$-junction bifurcation orbit}, $n>3$, if $x$ is a bifurcation orbit which splits into $k$ distinct periodic orbits as $t$ crosses into $(t_0,t_0+\epsilon)$ (similarly, for some $\epsilon>0$). In addition, we require any $k$-junction orbit to be decomposable into either a finite sum of period-doubling bifurcations or a finite sum of type $m\geq3$ bifurcations, as illustrated in Figure \ref{decomposition}. Either way, we require the Mallet-Yorke Indices of such a decomposition to include at most two indices - i.e., we cannot have three orbits in an $n$-junction whose Mallet-Yorke Indices are $1,-1$ and $0$.
\end{enumerate}

\begin{remark}
    Type $m$ bifurcation orbits are sometimes called "period multiplying bifurcations", originally discussed in \cite{Mey} and \cite{mey2} in the context of area-preserving maps. In addition, see also \cite{mey3} for the specific case of period tripling.
\end{remark}
\begin{figure}[h]
\centering
\begin{overpic}[width=0.4\textwidth]{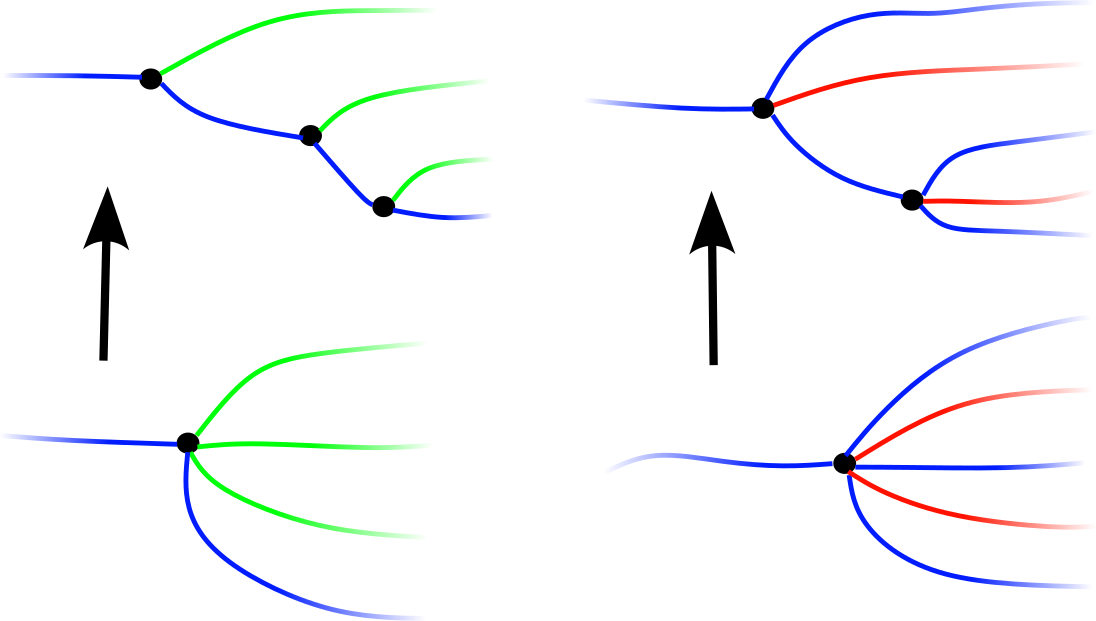}
\end{overpic}
\caption{\textit{On the left - a $4$-junction (on the bottom) which is decomposed into four period-doubling bifurcations. On the right - a $5$-junction (on the bottom) which is decomposed into two type $m$ bifurcations. The color blue represents Mallet-Yorke Index $1$, red represents the index $-1$ and green -- $0$.}}
\label{decomposition}
\end{figure}
Before moving on, we note that type $0$ periodic orbits can undergo a Neimark-Sacker bifurcation which does not change their Mallet-Yorke Index. This occurs, for example, when two complex conjugate eigenvalues $\lambda_1,\lambda_2$ cross $\{z||z|<1\}$ to $\{z||z|>1\}$ through $e^{\pm i\theta}$, where $\theta$ is irrational. We further note that type $0$ orbits are also isolated, i.e., if $x$ is a type $0$ periodic orbit of minimal period $n$ for some $f:S\to S$, there exists some neighborhood $N$ of $x$ s.t. $\{x\}=\{y|f^n(y)=y\}\cap N$. Moreover, by the Implicit Function Theorem it follows that if $f=f_\frac{1}{2}$ for some $C^1$ one-parameter family $f_t:S\to S$, $t\in(0,1)$, then there exists some $\epsilon>0$ s.t. if $\alpha=N\times(\frac{1}{2}-\epsilon,\frac{1}{2}+\epsilon)\cap \{(y,t)|f^n_t(y)=y\}$, then $\alpha$ is a collection of $n$ curves, each homeomorphic to $(0,1)$, passing through every point in $x$. As a consequence, if $\{x_j\}_j$ are a sequence of periodic orbits for $f$ s.t. $x_j\to x$, then the minimal period of $x_j$ must blow up as $j\to\infty$.

Throughout this paper, we will refer to periodic orbits of type $m\geq1$ as the \textbf{admissible bifurcations of periodic orbits} or just the \textbf{admissible bifurcations} in short. With that terminology, we now recall the connection between the types of periodic orbits, the admissible bifurcations and the Mallet-Yorke Index, given by the following result, proven in \cite{PY3}:
\begin{theorem}
\label{indexinvariance} Let $f_t:S\to S$ be a $C^1$ one-parameter family, $t\in[0,1]$, and let $t_0\in[0,1]$ be some interior point. Assume $x$ is a type $0$ periodic orbit of minimal period $n$ for $f_t$, which persists in $[0,t_0)$ and splits into $k$ distinct type $0$ periodic orbits for $f_t$, as $t$ crosses into $(t_0,1]$ (in particular, $t_0$ is bifurcation parameter for $x$). Let us denote by $x_1,...,x_k$ these new periodic orbits - then, we have the following invariance rule:
\begin{equation*}
    \phi(x)=\sum_{i=1}^k\phi(x_i)
\end{equation*}
In other words, the Mallet-Yorke Index is a bifurcation invariant (see the illustration in Figure \ref{fig1}).
\end{theorem}

Having recalled the Mallet-Yorke Index and its bifurcation invariance properties, we are ready to begin. As stated in the Introduction, we are interested in measuring the size of different bifurcation scenarios within the space of all "admissible" bifurcations - which for us will be the type $1,2, m$ and $n$-junctions orbits (where $m,n\geq3$). We now use this idea to introduce the following families of isotopies:

\begin{itemize}
    \item The set $F_1(S)$ is defined as the collection of $C^1$ one-parameter families in $S\times(0,1)$ s.t. every periodic orbit is either type $0$, type $1$ or type $2$ (i.e., every periodic orbit is allowed only to undergo saddle-node, Neimark Sacker or period-doubling bifurcations). In particular, the Mallet-Yorke Index is defined as either $0,1,-1$ for all type $0$ orbits - moreover, the Mallet-Yorke Index of a periodic orbit can change only at bifurcation orbits.
    \item The set $F_2(S)$ is defined as the collection of $C^1$ one-parameter families in $S\times(0,1)$ whose type $0$ periodic orbits are only allowed to undergo saddle-node, period-doubling and type $m\geq3$ bifurcations. In particular, we will assume that if a Neimark-Sacker bifurcation occurs where new periodic orbits are born, it has to be at a type $m\geq3$ bifurcation (note that families in $F_2(S)$ are technically allowed to have non-admssible periodic orbits, even if we ignore them). 
    \item The set $F_k(S)$, $k\geq3$, is defined as the collection of $C^1$ one-parameter families in $S\times(0,1)$ whose type $0$ periodic orbits bifurcate either by saddle-node, period-doubling, type $m\geq3$ or $r$-junction bifurcations, where $3<r\leq n+1$. Again, we assume that if a Neimark-Sacker bifurcation occurs in which new periodic orbits are born, it has to be at either type $m\geq3$ bifurcation or an $n$ junction, for some $k+1\geq n>3$. Similarly, families in $F_k(S)$ are also allowed to have periodic orbits that are not admissible.
\end{itemize}

By definition, for all $1\leq k<j$ we have $F_k(S)\subseteq F_j(S)$. As $F_1(S)$ is $C^1$-generic, it follows that each $F_k(S)$ is $C^1$-generic in the collection of $C^\infty$-isotopies $F:S\times(0,1)\to S$ - for the proof of genericity of $F_1(S)$, see Proposition 2.1.1 in \cite{EY} (for further discussion of the $C^r$, $r\geq1$ genericity of $F_1(S)$ see the Appendix of \cite{PY2}, Section 2 of \cite{KY2}, and the references therein). Before moving on, we make two remarks - the first one is that it is not known whether $F_k(S)$ is non-empty for an arbitrary $k>2$, which may lead one to wonder why we do not restrict ourselves to $F_1(S)$ and $F_2(S)$. The reason we incorporate these families into this paper is to account for possible degenerate bifurcations which are currently unknown -  as observed by several studies, it is unlikely one could ever classify all possible ways in which dynamics turn from simple to complex (see, for example, \cite{facet}). That being said, by the results of \cite{Quad} we know there are indications of $F_3(S)$ being non-empty, at least when $S$ is even-dimensional. The second remark is that one might wonder if it is possible to relax the assumptions and also consider the bifurcations of orbits that are not type $0$. Unfortunately, that is probably impossible. To illustrate, let $f_t:S\to S$, $t\in[0,1]$ denote a $C^1$ one-parameter family over a closed surface $S$ and assume $f_{\frac{1}{2}}$ has a fixed point $x$ which is \textbf{not} a type $m\geq3$ bifurcation orbit, yet $D_{f_{\frac{1}{2}}}(x)$ has a pair of imaginary eigenvalues that are roots of unity. The said fixed point would persist as we vary $t$, but could also lie at the center of a disc, which is rotated by some rational number. The bifurcations of the one-parameter family around such fixed points could probably be much more complex than what can be described using our methods. As such, we avoid this scenario by only considering type $0$ orbits and their bifurcations. 

Moving on, we will study the question of how much $F_k(S)$ fills up $F_j(S)$, $0<k<j$, by thinking of the bifurcation diagrams of type $0$ orbits  as a collection of graphs. To this end, we will first study the dependence of the admissible bifurcations on $d$, the dimension of $S$ - more precisely, we will consider how $d$ and the Mallet-Yorke Index constrain the bifurcations a type $0$ periodic orbit can undergo. To begin, we first introduce the following notation: consider a $C^1$ one-parameter family $f_t:S\to S$, $t\in(0,1)$ with a periodic orbit $x$ of type $0$ whose Mallet-Yorke Index is $i$. Assume, for simplicity, that $x$ persists as a type $0$ orbit for $f_t,t\in(0,t_1)$ and that at $t_1<1$ it bifurcates and splits into $k\in\mathbb{N}$ distinct type $0$ orbits of Mallet-Yorke Indices $i_1,...,i_k$ for $f_t$, $t\in(t_1,1)$. For simplicity, we denote this transition by $i\to (i_1,...,i_k)$. By Theorem \ref{indexinvariance}, we already know $i=\sum_{j=1}^k i_j$. So, we prove:
\begin{lemma}
    \label{notype20} Assume $x\in S$ is type $0$ periodic orbit for $f_{t_0}$, $t_0\in(0,1)$, and $\phi(x)=0$. Then, for all dimensions $d$, $x$ cannot split in a period-doubling bifurcation of the type $0\to(0,0)$ as we vary $t$ towards $1$.
\end{lemma}
\begin{proof}
   Assume $x$ has a minimal period $n$ and recall that $\phi(x)=0$ when $\sigma^-$, the number of multipliers of the differential $D_{f_{t_0}}(x)$ in $(-\infty,-1)$ is odd. Now, assume $x$ undergoes a period-doubling bifurcation at some $t_1>t_0$, where $x$ splits to $x'$ and $x''$ s.t. the period is doubled at, say, $x''$. We will prove the lemma by counting the eigenvalues of $D_{f^{2n}_t}(x'')$, $t>t_1$, in $(-\infty,-1)$ and showing it is even - this will imply the Mallet-Yorke Index of $x''$ is non-zero, thus by Theorem \ref{indexinvariance} the Mallet-Yorke Index of $x$ and $x'$ cannot both be $0$.
  
Again, assume $x$ goes through the period-doubling bifurcation at some $t_1>t_0$. As $-1$ is the only root of unity of $D_{f^{n}_{t_1}}(x)$ and since every negative eigenvalue for $D_{f^n_{t_1}}(x)$ becomes positive for $D_{f^{2n}_{t_1}}(x)$, the differential $D_{f^{2n}_{t}}(x'')$, $t>t_1$ cannot inherit any negative eigenvalues from $D_{f^n_{t_1}}(x)$ (at least when $t$ is close to $t_1$). Moreover, since complex conjugate eigenvalues for $D_{f^n_{t_1}}(x)$ only come in pairs, it follows that $D_{f^{2n}_{t_1}}(x)$ has at most an even number of negative eigenvalues - consequentially, the same is true for $D_{f^{2n}_{t}}(x'')$. 
\end{proof}

\begin{figure}[h]
\centering
\begin{overpic}[width=0.5\textwidth]{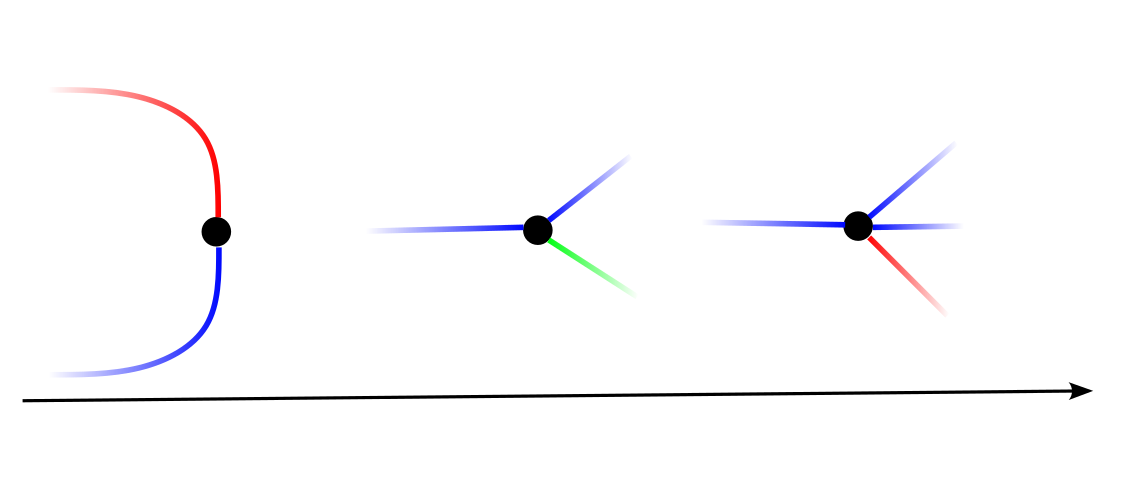}
\put(985,85){$t$}
\end{overpic}
\caption{\textit{The admissible saddle-node, period-doubling and type $m\geq3$ bifurcations when  $dim(S)=2$, sketched as diagrams. On red arcs the index is $-1$, on blue the index is $1$, while green arcs denote index $0$. The black dots denote the bifurcation orbit. }}
\label{fig3}
\end{figure}

Having proven Lemma \ref{notype20}, we are now ready to describe the admissible bifurcations when $dim(S)=d=2$. Namely, we prove the following:
\begin{proposition}
\label{rules2}    Consider a $C^1$ one-parameter family $f_t:S\to S$, $t\in(0,1)$ and let $x\in S$ be a periodic orbit of type $0$ for $f_{t_0}$, $t_0\in(0,1)$ of minimal period $n$. Then, when $dim(S)=2$, the saddle-node, period-doubling and type $m\geq3$ bifurcations for $x$ can only have one of the following forms (see Figure \ref{fig3}):
    \begin{itemize}
        \item  If $\phi(x)=-1$, $x$ cannot undergo any period-doubling or type $m\geq3$ bifurcations as we vary $t\in(0,1)$. In other words, $x$ can undergo type $2$ or type $m\geq3$ bifurcation only when $\phi(x)=1$.
        \item If $\phi(x)=0$, it cannot split into three periodic orbits in a type $m\geq3$ bifurcation as we vary $t$.
        \item Moreover, at type $m\geq3$, $x$ can only split by the rule $1\to(1,-1,1)$. Consequentially, the only type $m\geq3$ bifurcations allowed are $1\to(1,-1,1)$.
    \end{itemize}

As a consequence, any $k$-junctions, $k>3$ has one of the following types (see Figure \ref{rules2d}):
     \begin{itemize}
         \item If  an $kn$-junction is decomposed into a finite number of period-doubling bifurcations, then $k-1$ orbits have Mallet-Yorke Index $0$, and one has Mallet-Yorke Index $1$.
         \item If an $3j=n$-junction is decomposed into $j$ type $m\geq3$ bifurcations, then $2j$ orbits have Mallet-Yorke Index $1$, while $j$ have Mallet-Yorke Index $-1$.
     \end{itemize}
\end{proposition}
\begin{proof}
    Consider the differential of $D_{f^n_{t_0}}(x)$ - if $\phi(x)=-1$, by the orientation-preserving assumption on the one--parameter family $f_t:S\to S$, $t\in[0,1]$, we conclude that the differential has one eigenvalue in $(0,1)$ and another in $(1,\infty)$. This proves that periodic orbits of Mallet-Yorke Index $-1$ cannot have their period-doubled or $m$-multiplied,  i.e., if $\phi(x)=-1$, $x$ cannot undergo period-doubling or type $m\geq3$ bifurcation, as the first requires two negative eigenvalues while the second requires two complex conjugate eigenvalues. Similarly, whenever $\phi(x)=0$, the differential must have two negative eigenvalues, hence periodic orbits of Mallet-Yorke Index $0$ cannot have their period $m$-multiplied in a type $m\geq3$ bifurcation.

\begin{figure}[h]
\centering
\begin{overpic}[width=0.35\textwidth]{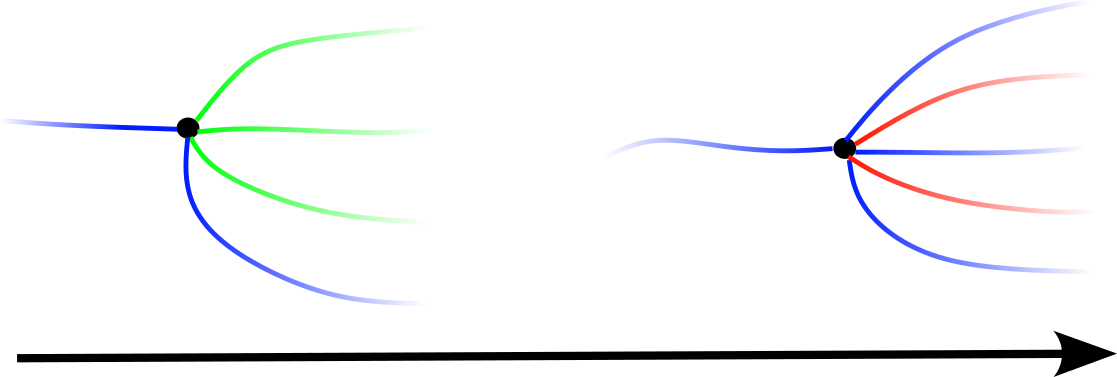}
\put(1030,10){$t$}
\end{overpic}
\caption{\textit{The admissible $n$-junctions in dimension $2$ for $n=4,5$. The colors denote indices as in Figure \ref{fig3}. }}
\label{rules2d}
\end{figure}

Therefore, it remains to prove that the only admissible type $m\geq3$ bifurcations in dimension $2$ are $1\to(-1,1,1)$. To do so, note that by the above arguments and by Theorem \ref{indexinvariance}, we can only have $1\to(-1,1,1)$ or $1\to(0,1,0)$ - where the branch of orbits, along which the period is preserved, continues with the original Mallet-Yorke Index, i.e., $1$ (see also Section 1 in \cite{Mey}). Taking the $mn$-th iterate (i.e., the least common iterate), it is easy to see that at the type $m$ bifurcations the two imaginary eigenvalues can only change into positive eigenvalues, as they are $m$-th roots of unity. This proves that the two bifurcating branches of orbits whose period is $m$-multiplied both have real, positive eigenvalues, hence by Theorem \ref{indexinvariance} the only possibility is that one has Mallet-Yorke Index $1$ and the second $-1$. This proves that the only admissible type $m\geq3$ bifurcation is $1\to(1,-1,1)$. It is easy to see that the implications of this classification of type $m$ and period-doubling bifurcations for $n$-junctions are immediate. 
    \end{proof}
\begin{remark}
    The fact that when $\phi(x)=-1$, $x$ cannot undergo a period-doubling bifurcation was originally proven in \cite{KE} for three-dimensional flows, using Knot Theory. That being said, as shown in \cite{KE}, the same bifurcation most certainly can happen for four dimensional flows, where it is referred to as a "Noose Bifurcation".
\end{remark}

\begin{figure}[h]
\centering
\begin{overpic}[width=0.4\textwidth]{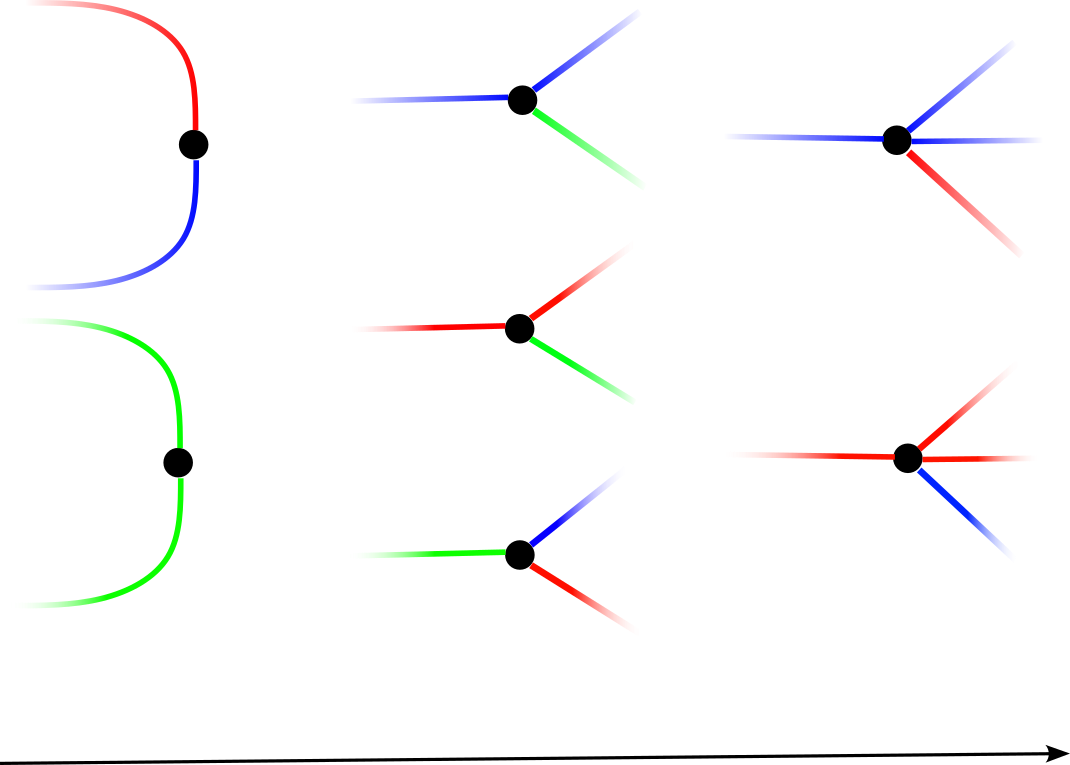}
\put(1030,0){$t$}
\end{overpic}
\caption{\textit{The admissible saddle-node, period-doubling and type $m\geq3$ when $dim(S)=3$, sketched as diagrams. On red arcs the index is $-1$, on blue the index is $1$, while green arcs denote index $0$. The black dots denote the bifurcation orbit.}}
\label{fig4}
\end{figure}

Having studied the case when $dim(S)=2$, we now move on and study the admissible bifurcation laws when $dim(S)=3$. We first remark that  every admissible bifurcation that occurs in dimension $2$ can also occur in dimension $3$ - all we need to do is add one eigenvalue inside the interval $(0,1)$ to the differential, which would not change the Mallet-Yorke Index along the diagram. That being said, when one moves from dimension $2$ to $3$ there are more admissible bifurcations possible in $F_2(S)$, as we will prove below:

\begin{proposition}
\label{rules3}    Assume $dim(S)=3$ and let $f_t:S\to S$, $t\in(0,1)$ be a $C^1$ one-parameter family. Then, the possible saddle-node, period-doubling and type $m\geq3$ bifurcations for type $0$ orbits can only have the following forms  (see Figure \ref{fig4}):
\begin{enumerate}
    \item A saddle-node bifurcation, where two periodic orbits of Mallet-Yorke Index $0$ collide with one another and disappear. 
    \item A period-doubling bifurcation of either one of the forms - $0\to(-1,1)$, $-1\to(-1,0)$.
    \item Type $m\geq3$ bifurcations of the form $-1\to(-1,1,-1)$ or $1\to(1,-1,1)$.
\end{enumerate}
In other words, the possible bifurcation laws when $S$ is three-dimensional are precisely the admissible rules from the two-dimensional case in Proposition \ref{rules2} plus the additional $5$ forms of bifurcations given above. Consequentially, the possible $n>3$-junctions are as in Figure \ref{rules3d} - or every $n$-junction takes one of the following forms:
\begin{itemize}
    \item A periodic orbit of Mallet-Yorke Index $\pm1$ which splits into $n-1$ distinct periodic orbits of Mallet-Yorke Index $0$, and one periodic orbit of Mallet-Yorke Index $\pm1$.
    \item A periodic orbit of Mallet-Yorke Index $\pm1$ which splits either into $2j$ periodic orbits of Mallet-Yorke Index $-1$ and $j$ periodic orbits of Mallet-Yorke Index $1$, or the opposite.
\end{itemize}
\end{proposition}
\begin{proof}
Let $x$ be a periodic orbit - for simplicity, we assume it is a fixed point. We study the admissible bifurcations it can undergo depending on whether $\phi(x)=0,1,-1$. In each case, we denote the eigenvalues of the differential by $(\lambda_1,\lambda_2,\lambda_3)$ and recall that as the one-parameter family $f_t:S\to S$, $t\in(0,1)$, is an orientation-preserving isotopy, we always have $\lambda_1\cdot\lambda_2\cdot\lambda_3>0$. Therefore, if there are negative multipliers, there are precisely two - and if there are complex-conjugate multipliers, again, there are exactly two. In particular, whatever the case, there always exists at least one positive eigenvalue.

We prove the statement on a case-by-case basis. We denote the multipliers by their values. Therefore, for brevity, we now introduce the following notations:
\begin{itemize}
    \item $\lambda=\alpha$ if $\lambda\in(0,1)$.
    \item $\lambda=\beta$ if $\lambda>1$,
    \item $\lambda=\gamma$ if $\lambda\in(-1,0)$.
    \item $\lambda=\delta$ if $\lambda<-1$.
\end{itemize}

We will often write $(\lambda_1,\lambda_2,\lambda_3)$ have the form $(r,q,z)$, where $r,q,z\in\{\alpha,\beta,\gamma,\delta\}$ to denote their type, as described above. We begin with the saddle-node case, where, without any loss of generality, we assume $\lambda_1>0$. It is easy to see that when $\lambda_2=\gamma$, $\lambda_3=\delta$, then we can smoothly vary the values multiplier $\lambda_1$ along $(1-\epsilon,1+\epsilon)$ without changing $\lambda_2,\lambda_3$ s.t. when it hits $1$ there is a saddle-node bifurcation. As the values of $\lambda_3,\lambda_2$ are unchanged along this isotopy, it follows that the Mallet-Yorke Index remains $0$, i.e., it does not change at the saddle-node bifurcation. This proves the only additional possibility for a saddle-node bifurcation when $dim(S)=3$ is that two $0$-Index periodic orbits collide and vanish.

\begin{figure}[h]
\centering
\begin{overpic}[width=0.35\textwidth]{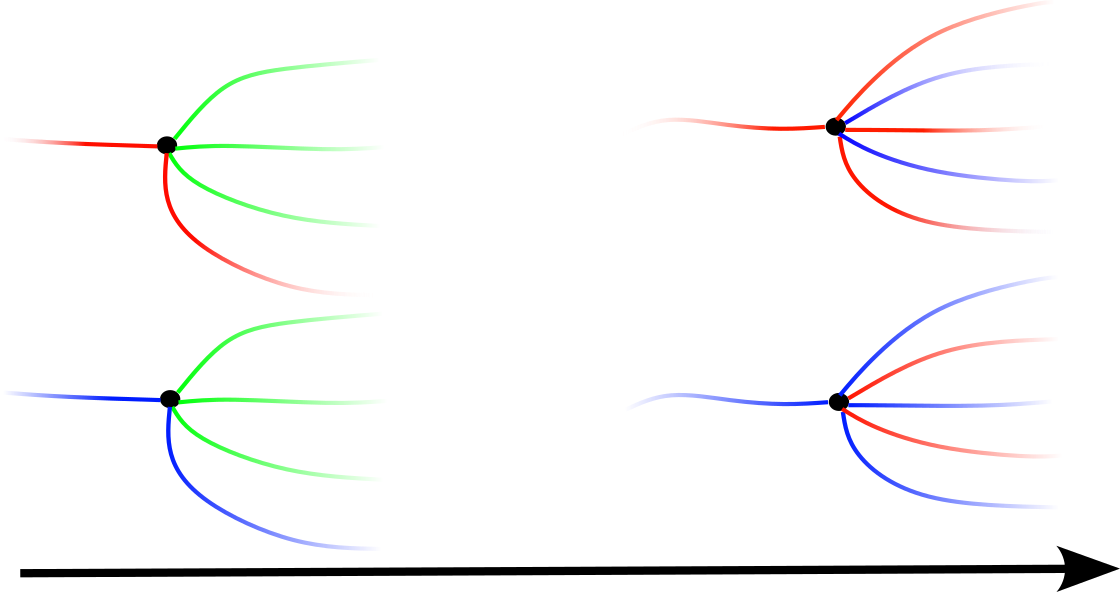}
\put(1020,5){$t$}
\end{overpic}
\caption{\textit{The admissible $n$-junctions in dimension $3$ for $n=4,5$. Again, the colors denote the indices of the corresponding orbits.}}
\label{rules3d}
\end{figure}

We now study the period-doubling case. We already know $1\to(0,1)$ is possible from the two-dimensional case and that $0\to(-1,1)$ is possible by \cite{KE}. By Lemma \ref{notype20} we cannot have $0\to(0,0)$, therefore, it remains to show that the bifurcation $-1\to(0,-1)$ is also possible. To this end, note that if $\phi(x)=-1$ and $x$ undergoes a period-doubling in dimension $3$, then before the bifurcation occurs, $(\lambda_1,\lambda_2,\lambda_3)$ can only be of the form $(\beta,\gamma,\gamma)$ or $(\beta,\delta,\delta)$.

Therefore, we now proceed by performing a period-doubling bifurcation on $x$, $\phi(x)=-1$ as follows - we first recall that at the bifurcation, precisely one of the $\gamma$ or $\delta$ crosses over $-1$ to either $(-\infty,-1)$ or $(-1,0)$ (respectively), following which $x$ is split into $x',x''$, where $x''$ has its period-doubled. Moreover, we do so s.t. the eigenvalues $(\lambda_1,\lambda_2,\lambda_3)$ of the differential of $x'$, the non-doubled orbit, have the form $(\beta,\gamma,\delta)$ - as such, the Mallet-Yorke Index of $x'$ is $0$. Similarly, for the branch of $x''$ (i.e., the doubled branch), $(\lambda_1,\lambda_2,\lambda_3)$ have the form $(\beta,\alpha,\alpha)$ (the doubled branch, whose index is $-1$). All in all, it follows that $-1\to(0,-1)$ is a possibility when the dimension of $S$ is $3$.

We now consider type $m\geq3$ bifurcation. By the above, at such bifurcations $\lambda_1$ is real (and positive), while $\lambda_2,\lambda_3$ are conjugate $m$-th roots of unity. Therefore, only periodic orbits of Mallet-Yorke Index $\pm1$ can have their period $m$-multiplied. This implies that when $dim(S)=3$, type $m\geq3$ bifurcation can belong to at most one of the following options:
\begin{enumerate}
    \item $-1\to(1,-1,-1)$
    \item $-1\to(-1,0,0)$
    \item $1\to(-1,1,1)$,
    \item $1\to(0,1,0)$.
\end{enumerate}

Where, similarly to the two-dimensional case, by the Lefschetz Fixed Point Theorem we know the Mallet-Yorke Index survives along the branch of orbits whose period is \textbf{not} doubled by $m$. Now, assume the period of $x$ is $n$ and that it goes through a type $m\geq3$ bifurcation at $t_0\in(0,1)$ - by the arguments above, its Mallet-Yorke Index is either $1$ or $-1$. We further note that by both definition and continuity, all eigenvalues of $D_{f^{mn}_{t_0}(x)}$ at the bifurcation must be positive (since $\lambda_1>0$ and $\lambda_2,\lambda_3$ are roots of unity). This implies that there can be no $0$ orbits in the periodic orbits splitting from $x$ whose minimal periods are $mn$. In other words, the only possibilities are $-1\to(1,-1,-1)$ and $1\to(1,-1,1)$.

We now conclude the proof by considering the possible $n>3$-junctions. By the previous discussion, all the $k$-junctions admissible in the two-dimensional case are also admissible in the three-dimensional one. Combined with the analysis above and recalling the orbits bifurcating from $n$-junctions can be divided to at most two indices, a similar argument to Proposition \ref{rules2} yields the two forms in Figure \ref{rules2d}.
\end{proof}
\begin{remark}
    In accordance with previous remark, the $0\to(-1,1)$ period-doubling bifurcation discussed above is the one, corresponding to the Noose Bifurcation discovered in \cite{KE} for four dimensional flows.
\end{remark}
\begin{figure}[h]
\centering
\begin{overpic}[width=0.4\textwidth]{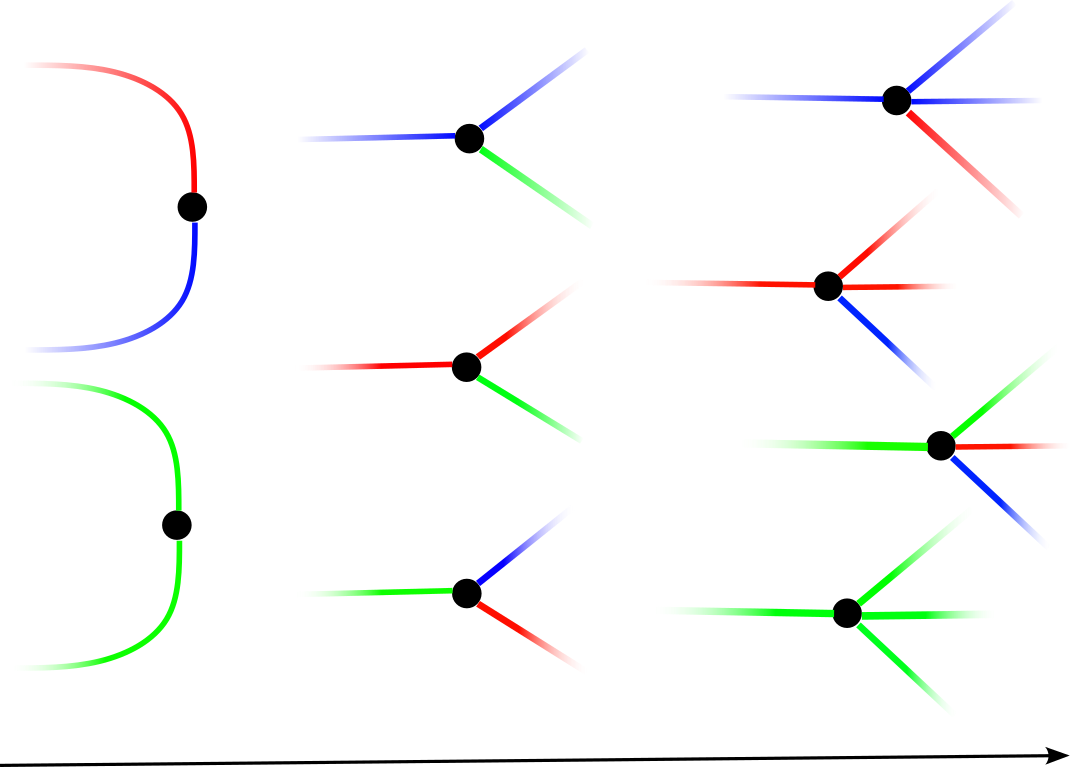}
\put(1020,0){$t$}
\end{overpic}
\caption{\textit{The admissible saddle-node, period-doubling and type $m\geq3$ when  $dim(S)\geq4$, sketched as diagrams. On red arcs the index is $-1$, on blue the index is $1$, while green arcs denote index $0$. The black dots denote the bifurcation points. }}
\label{fig5}
\end{figure}

Having studied the three-dimensional case, we finally study the case where $dim(S)=d\geq4$. Again, it is easy to see that by adding a finite number of eigenvalues in the interval $(0,1)$, all bifurcations admissible in dimension $3$ can also occur in dimension $4$ and higher. With that in mind, we now prove:

\begin{proposition}
    \label{rules4} When $dim(S)\geq4$, in addition to the bifurcation rules of dimensions $2$ and $3$, the only possible extra bifurcations for type $0$ orbits are two type $m\geq3$ bifurcations of the kind $0\to(0,0,0)$ and $0\to(0,-1,1)$ (see the illustration in Figure \ref{fig5}). Consequentially, the only additional $n$-junction has the type of an orbit splitting into $3j$ orbits, all of which have Mallet-Yorke Index $0$ (see the illustration in Figure \ref{rules4d}).
\end{proposition}
\begin{proof}
    By Lemma \ref{notype20} there is no period-doubling bifurcation with $0\to(0,0)$. Therefore, as all the other possible cases were already covered in dimensions $3$ and $2$, we need only consider the type $m$ bifurcations. Again, by the low-dimensional case we already know $-1\to(-1,1,-1)$ and $1\to(1,-1,1)$ are a possibility, so we need only consider the four cases:
    \begin{enumerate}
        \item $0\to(0,0,0)$.
        \item $0\to(0,-1,1)$.
        \item $-1\to(0,-1,0)$
        \item $1\to(0,1,0)$.
\end{enumerate}

We use the same notation with $\alpha,\beta,\gamma,\delta$ as in Proposition \ref{rules3}. Moreover, recall that if the eigenvalues of the differential at any periodic orbit are denoted by $\lambda_i$, $i=1,2,3,4$, then by the orientation-preserving assumption on the isotopy we have $\prod_{i=1}^n\lambda_i>0$. We first show $0\to(0,0,0)$ and $0\to(0,-1,1)$ can occur when $dim(S)=4$ - which will show that these two bifurcations can occur in all higher dimensions as well. We first prove that $0\to(0,0,0)$ can occur -  assume $x$ is a fixed point with multipliers $(\lambda_1,\lambda_2,\lambda_3,\lambda_4)=(\gamma,\delta,r,\overline{r})$ (where $r,\overline{r}$ are complex-conjugate), which implies $\phi(x)=0$. By splitting it into $3$ orbits in a type $3$ bifurcation we see that along the $0$ branch that doesn't change its period, the eigenvalues do not change their type, i.e., the Mallet-Yorke Index remains zero and the differential of this branch of orbits has a pair of complex-conjugate eigenvalues. By taking the third iterate, the multipliers change along each of the tripled branches into either $(\gamma^3,\delta^3,\alpha,\alpha)$, $(\gamma^3,\delta^3,\beta,\beta)$ or $(\gamma^3,\delta^3,\alpha,\beta)$. Since $\gamma^3,\delta^3$ remain in $(-1,0)$, $(-\infty,-1)$ (respectively), the Mallet-Yorke Index is $0$ along each branch as $\sigma^-=1$. This proves that $0\to(0,0,0)$ is a possibility.

\begin{figure}[h]
\centering
\begin{overpic}[width=0.35\textwidth]{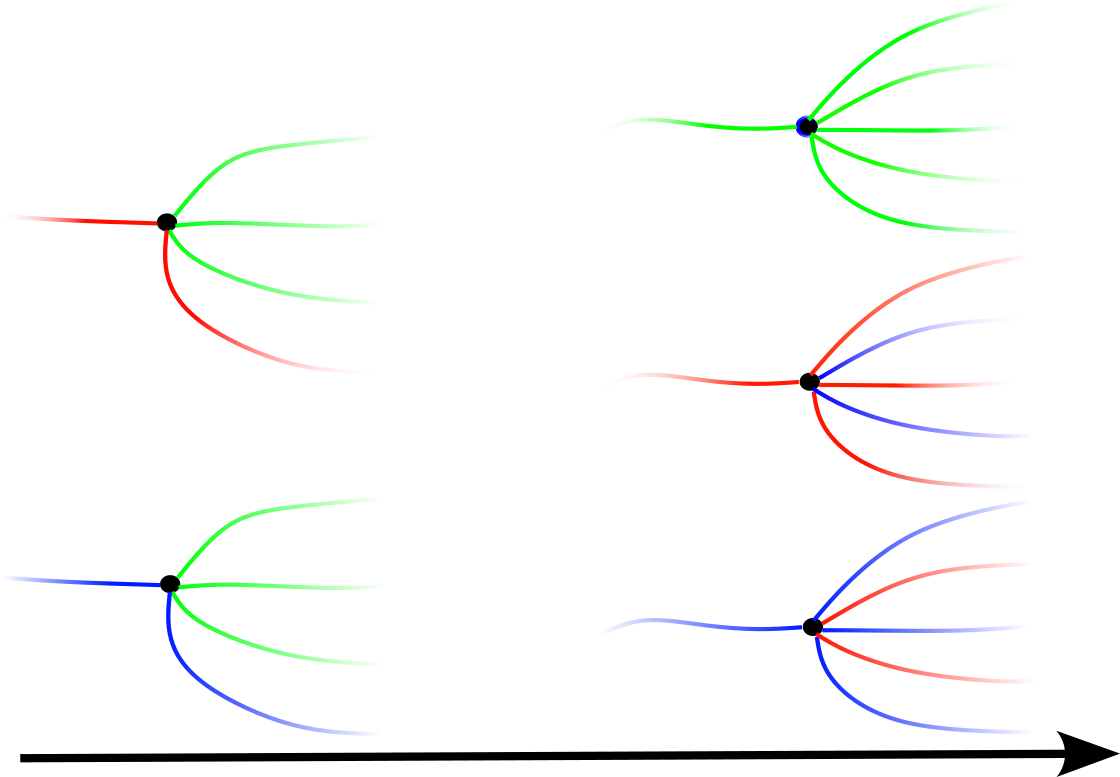}
\put(1020,10){$t$}
\end{overpic}
\caption{\textit{The admissible $n$-junctions in dimension $4$ for $n=4,5$. Again, the colors denote the indices of the corresponding orbits - red for $-1$, blue for $1$, and green for $0$.}}
\label{rules4d}
\end{figure}

We now show $0\to(0,-1,1)$ is also a possibility. Again, whenever this occurs, the Mallet-Yorke Index remains $0$ along the branch which is \textbf{not} $m$-multiplied. Now, take a fixed point $x$ with eigenvalues of the type $(\gamma,\delta,r,\overline{r})$ and perform a type $4$ bifurcation which changes the eigenvalues along the two period $4$ branches splitting from $x$ at the bifurcation to $(\gamma^4,\delta^4,\alpha,\alpha)$ and $(\gamma^4,\delta^4,\alpha,\beta)$ (this agrees with the Lefschetz number of the iterates). All in all, we see that $0\to(0,-1,1)$ is also a possibility.

We now conclude the proof by showing that at any dimension $d\geq4$, when $\phi(x)\in\{-1,1\}$ the orbit $x$ cannot split by the rule $\phi(x)\to(0,\phi(x),0)$ in a type $m\geq3$ bifurcation (where $n$ denotes the period of $x$). To this end, the eigenvalues of $D_{f_t^n}(x)$ which are complex $m$-th roots of numbers in $(-\infty,-1)$ always come in complex-conjugate pairs - where by complex $m$-th root we mean any eigenvalue for $D_{f^n}(x)$, which is \textbf{not} negative for $D_{f^n}(x)$, yet becomes negative for $D_{f^{mn}}$ (by definition, such eigenvalues cannot be real).

Now, assume $x$ splits to $x',x'',x'''$ at a type $m\geq3$ bifurcation at some $t_0\in(0,1)$ s.t. as we cross into $t>t_0$ the period is preserved at $x'$ and $m$-multiplied at $x''$ and $x'''$. If $\sigma^-$ is even for $x$ (which is the case when $\phi(x)\in\{-1,1\}$), it remains so for $x'$. Moreover, it follows that $\sigma^-$ also remains even for $x''$ and $x'''$ due to the following:
\begin{itemize}
    \item Negative eigenvalues for $D_{f_{t_0}^n}$ in $(-\infty,-1)$ either remain all negative or become all positive for $D_{f_t^{mn}}$, $t>t_0$ - depending on whether $m$ is odd or even. Therefore, the number of these "inherited" eigenvalues, $n_1$, either becomes $0$ or stays the same, hence, even.
    \item As the complex $m$-th roots of unity come in pairs, they add even number of eigenvalues $n_2$ in $(-\infty,-1)$.
\end{itemize}

This implies $D_{f_t^{mn}}, t>t_0$, has $n_1+n_2$ eigenvalues in $(-\infty,-1)$, hence, $\sigma^-=n_1+n_2$ is even. In other words, neither $x''$ or $x'''$ can have an odd number of eigenvalues in $(-\infty,-1)$ w.r.t. $D_{f_t^{mn}}$. Hence, a splitting $\phi(x)\to(0,\phi(x),0)$, $\phi(x)=\pm1$ is impossible, as a $0$ branch must have $\sigma^-$ odd. As a consequence, all the $n>3$-junctions are as in Figure \ref{rules4d} - i.e., in additional to all the possible three-dimensional $k$-junctions, the only extra possibility is an $n$-junction whose orbits all have Mallet-Yorke Index $0$.
\end{proof}
\begin{remark}
    At this point we remark that the proofs of Propositions \ref{rules4}, \ref{rules3} and \ref{rules2} only show which bifurcations of periodic orbits \textbf{cannot be ruled out topologically} in $F_k(S)$, $k>0$ - not that they necessarily occur along some isotopy. In other words, the propositions give a "topological upper bound" on the complexity of bifurcations.
\end{remark}
Having studied the admissible bifurcations of periodic orbits in $F_1(S)$ and $F_2(S)$ (and to a lesser extent, also $F_k(S)$, $k\geq3$), we are now ready to make the transition from one-parameter families into colored graphs. To this end, again, let $S$ be a closed smooth manifold of dimension $d\geq2$, and let $f_t:S\to S$, $t\in(0,1)$ be a $C^1$-one parameter family in $F_k(S)$, $k\geq1$. Let $Per\subseteq S\times(0,1)$ denote the collection of points satisfying the following:

\begin{itemize}
    \item Every $(x,t)\in Per$ is periodic for $f_t$, i.e., there exists some minimal $n>0$ s.t. $f^n_t(x)=x$.
    \item Every $(x,t)$ lies on a periodic orbit for $f_t$ that is either type $0$, type $1$, type $2$, type $m\geq3$, or a $n$-junction orbit (for some $n\geq4$). Moreover, given any component $C$ of $Per$, for a.e. $t$, $(y,t)$ lies on a type $0$ periodic orbit $(x,t)$ for $f_t$. 
    \item  There are no components of $Per$ that are singletons - i.e., all components are branched $1$-manifolds.  
    \item Finally, if $C$ is a component of $Per$ and $(x,t)=C\cap S\times\{t\}$ is a branching point for $C$ (for some $t$), then $x$ is either a type $2$ bifurcation, a type $m\geq3$ bifurcation, or an $n\geq4$-junction.
\end{itemize}

In other words, the set $Per$ is composed \textbf{only} from the periodic orbits which we admit in $\cup_{j}F_j(S)$ and their bifurcations. In particular, if $C$ is a component of $Per$ and $(x,t)=C\cap S\times\{t\}$ is a branching point where $C$ fails to be an arc, then the following holds:
\begin{itemize}
    \item If $C\setminus(x,t)$ includes three components, $x$ lies on a period-doubling orbit for $f_t$.
    \item If $C\setminus(x,t)$ includes four components, $x$ lies on a type $m$ periodic orbit for $f_t$ for some $m\geq3$.
    \item If $C\setminus (x,t)$ has more than four components, $x$ lies on some $n$-junction for $f_t$ for some $k+1\geq n\geq4$.
\end{itemize}

We note that $F(C)=C'$ for some $C'$ a component of $Per$, where $F:S\times(0,1)\to S\times(0,1)$ is defined by $F(x,t)=(f_t(x),t)$. Moreover, $F(C)=C$ precisely when $C$ includes a fixed point for some $f_t$, $t\in[0,1]$. We now define an equivalence relation on the components of $Per$, where $C\sim C'$ if there exists some $n$ s.t. $F^n(C)=C'$ - which makes every equivalence class $[C]$ into a branched curve, where, again, the branches correspond to the bifurcation orbits. We now prove the following easy Lemma:
\begin{lemma}
\label{colorlemma}    $[C]$ is a graph $(V',E')$, whose set of vertices $V'$ corresponds to the branching points and its edges $E'$ are curves (see the illustration in Figure \ref{nodrawback}). Moreover, we can canonically color the edges in $E'$ based on the Mallet-Yorke Index.
\end{lemma}
\begin{proof}
 It is easy to see $[C]$ is a graph - it remains to describe how it can be colored. By $C\subseteq Per$ we know that given any edge $e\in E'$, we can define the Mallet-Yorke Index on $e$. To see why, recall $e$ corresponds to some arc $c$ on $C$, composed of points that lie on type $0$ periodic orbits. Whenever $(y,t)\in c$ we know $(y,t)$ lies on some periodic orbit $x$ for $f_t$ which, by definition, has to be type $0$ - hence, its Mallet-Yorke Index is well defined as either $\pm1$ or $0$. The Mallet-Yorke Index is independent of our choice of component $C\in[C]$, and since by Theorem \ref{indexinvariance} we know the Mallet-Yorke Index is constant as we vary $(y,t)$ along $c$, we can assign the same index, i.e., $-1,1$ or $0$ to $e$. All in all, it follows we can assign a color to every edge $e\in E'$ - say, red for $-1$, blue for $1$, and green for $0$ (see the illustration in Figure \ref{nodrawback}). 
\end{proof}
\begin{remark}
    Some doubts may be raised about whether $c$ truly is an arc on $C$. To dispel them, note that with previous notations, if $n$ is the minimal period of $x$ then as $x$ is type $0$ and $y$ lies on $x$, the differential $D_{f^n_t}(y)$ is an invertible matrix. By the Implicit Function Theorem this implies the solutions to the equation $g(y,r)=f^n_r(x)$ form a curve in $S\times(0,1)$ passing through $(y,t)$ - at least when $|r-t|$ is sufficiently small. Since $(y,t)$ is some arbitrary point on $c$ the assertion follows.
\end{remark}
With that geometric picture in mind, we now proceed to define our notion for a bifurcation diagram:
\begin{definition}
    \label{bifurcationgraph} Let $f_t:S\times[0,1]\to S$ be a $C^1$ one-parameter family of diffeomorphisms in $F_k(S)$. With previous notations, its \textbf{admissible bifurcation diagram}, or in short, \textbf{bifurcation diagram} is a colored graph $\Gamma=(V,E)$ defined as follows (see the illustration in Figure \ref{nodrawback}):
    \begin{itemize}
        \item Each equivalence class $[C]$ corresponds to precisely one component of $\Gamma$ - a subgraph denoted by $\gamma$, where $\gamma=(V',E')$ as above.
       \item For every component $\gamma\subseteq \Gamma$, we color the edges $E'$ of $\gamma=(V',E')$ as described in Lemma \ref{colorlemma}, where blue corresponds to Mallet-Yorke Index $1$, red to $-1$ and green to $0$ (see the illustration in Figure \ref{nodrawback}).
    \end{itemize}

Note that since we assume the one-parameter family is in $F_k(S)$ - i.e., a periodic orbit is allowed to split to at most $k+1$ periodic orbits - it immediately follows the vertex of every graph $\gamma\in \Gamma$ has valence \textbf{at most} $k+2$. In particular:
\begin{itemize}
    \item A vertex corresponding to a saddle node has degree 2,
    \item A vertex corresponding to period doubling has degree 3.
    \item A vertex corresponding to type $m\geq3$ bifurcation has degree 4.
    \item A vertex corresponding to a $j$-junction, $k+1\geq j>3$ has degree $j+1$.
\end{itemize}
\end{definition}
 \begin{figure}[h]
\centering
\begin{overpic}[width=0.2\textwidth]{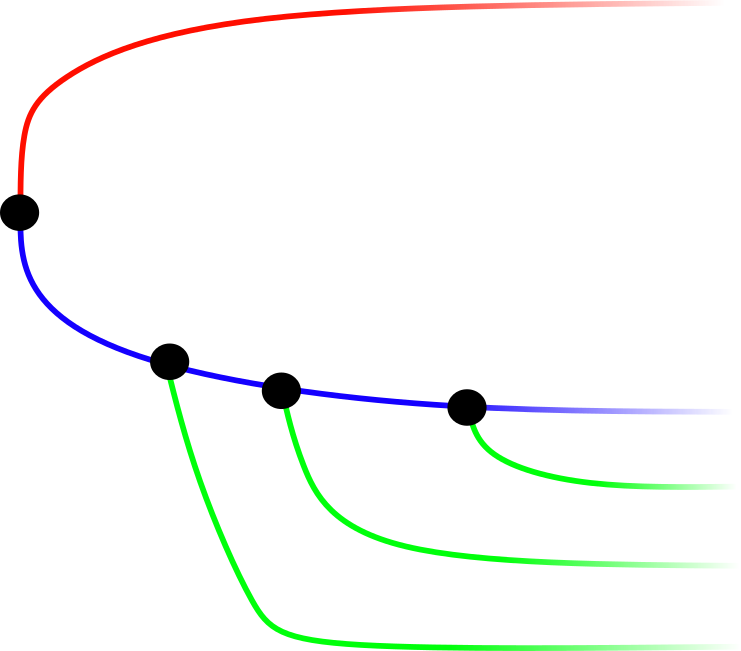}
\end{overpic}
\caption{\textit{A (possible) bifurcation diagram in $F_1(S)$ including one vertex corresponding to a saddle-node bifurcation and three others corresponding to period-doubling bifurcations.}}
\label{nodrawback}
\end{figure}

Having defined the bifurcation diagram of a given one-parameter family $f_t:S\to S$, $t\in(0,1)$ in $F_k(S)$, $k\geq1$, we now move to discuss the bifurcation diagrams corresponding to one-parameter families in $F_j(S)$ as a collection of graphs. To this end, we define the following:
\begin{definition}
    \label{graphset} Let $S$ be a smooth closed manifold of dimension $d\geq2$ and let $F_k(S)$ be as before. Then, $B_k(S)$, the \textbf{bifurcation set of $F_k(S)$,} will denote a collection of  colored graphs satisfying the following:
    \begin{itemize}
        \item There exist a map $\pi:F_k(S)\to B_k(S)$ assigning each curve in $F_k(S)$ a bifurcation diagram.
        \item If $\Gamma\in B_k(S)$, then $\pi^{-1}(\Gamma)$ is non-empty.
    \end{itemize}
\end{definition}
\begin{figure}[h]
\centering
\begin{overpic}[width=0.3\textwidth]{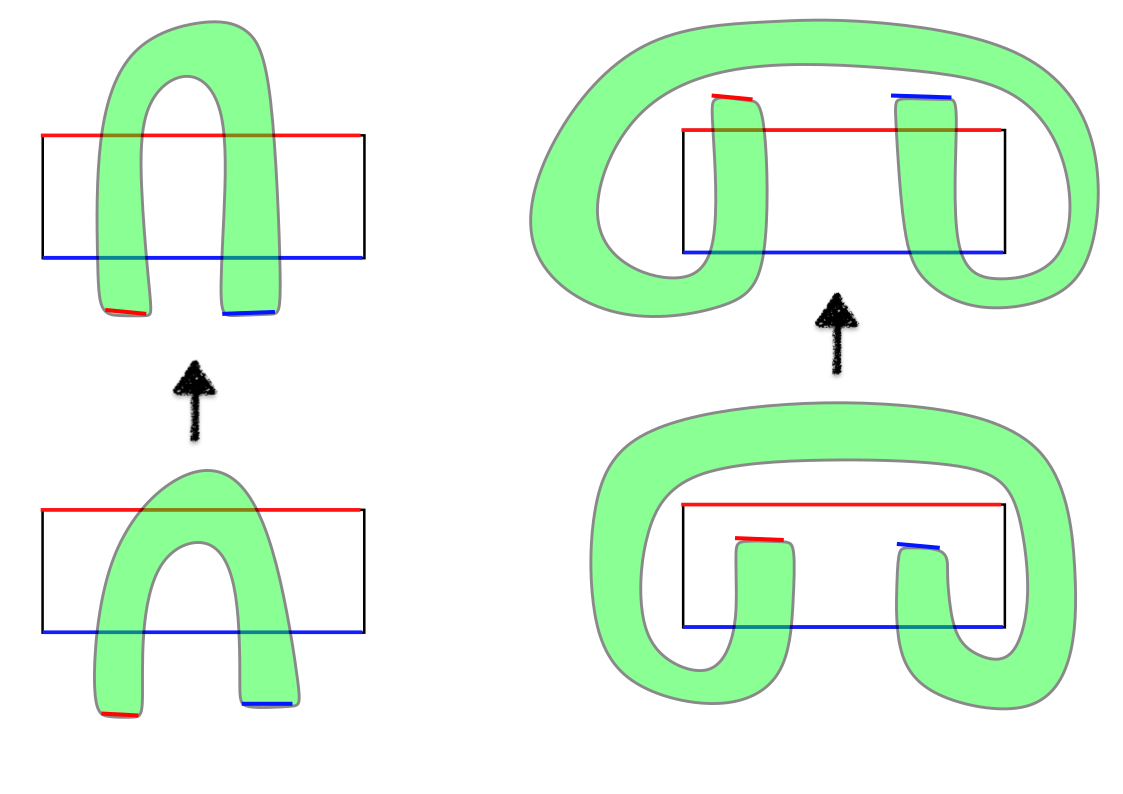}
\end{overpic}
\caption{\textit{Two illustrations of a rectangle isotopy progressing into a Smale horseshoe.}}
\label{horseshoes}
\end{figure}

In other words, $B_k(S)$ is a "moduli space" of possible bifurcation diagrams. Note that in general, there is no good reason to assume $\pi$ is injective - for example, consider to isotopies creating different horseshoes, as illustrated in Figure \ref{horseshoes}. By \cite{KY4} we know that under certain conditions, periodic orbits in both would be expected to appear via a period-doubling cascades of attractors. If $\Gamma_1$ and $\Gamma_2$ are the bifurcation diagrams corresponding to these isotopies, it may well be that $\Gamma_1=\Gamma_2$. 

The main conclusion from the discussion above is that in order to study the bifurcations of periodic orbits for one-parameter families in $F_k(S)$ it would suffice to study the graphs in $B_k(S)$. Unfortunately, this approach is too naive to work, even in low dimensions - specifically, the bifurcation diagram of a one-parameter family of surface diffeomorphisms $f_t:S\to S$, $t\in(0,1)$ may well depend on the isotopy class of, say, $f_{\frac{1}{2}}$ (for a survey of these ideas, see \cite{Bo}). As such, one probably cannot hope to begin developing a general theory based on the $B_k(S)$, $j\geq1$ alone. To deal with this difficulty, we instead consider a family of graphs $G_{k,d}$ which, even though it is possibly larger than $B_k(S)$, can be analyzed easily. To do so, let $G_{k,d}$ denote the collection of all possible colored graphs satisfying the following (see Figures \ref{fig3}, \ref{fig4} and \ref{fig5}):
\begin{itemize}
    \item If $\Gamma\in G_{k,d}$, every vertex of $\Gamma$ has degree at most $k+2$.
    \item Recalling we identify the Mallet-Yorke Indices $1,-1$ and $0$ with the colors blue, red and green (respectively), the graphs of $G_{k,d}$ are colored s.t. Theorem \ref{indexinvariance} is satisfied.
    \item Depending on $d$, the colorings of graphs in $G_{k,d}$ also obey the bifurcation rules dictated by Propositions \ref{rules2}, \ref{rules3} and \ref{rules4} (respectively).
\end{itemize}

Thus, $G_{k,d}$ is the collection of all colored graphs which could be created by the admissible bifurcation rules in $F_k(S)$, when treated abstractly as "recipes" for creating bifurcation diagrams. It is easy to see that by definition, for every closed, smooth $S$ of dimension $d$ we have $B_k(S)\subseteq G_{k,d}$. 

This definition of $G_{k,d}$ naturally leads us to the following problem: for which manifolds $S$ should we expect $G_{k,d}=B_k(S)$? We will not address this problem in this paper, as it is beyond its scope - and in fact, due to the deep connection between periodic orbits and the topology of the manifold on which they lie, it is very likely there is no general answer to this question. That being said, one can construct explicit examples where $B_k(S)\ne G_{k,d}$, which we now discuss (see also Appendix \ref{everygraph}).

To illustrate such an example, consider the case when $S=S^2$ and $k=1$ - we are now going to show one can have admissible diagrams in $G_{1,2}$ which cannot occur in $F_1(S^2)$ (although they cannot be ruled out as a bifurcation subdiagram for a larger set). In detail, consider a diagram which includes a period-doubling bifurcation where a $1$-Mallet-Yorke Index orbit splits into a orbits with $0$ and $1$ Mallet-Yorke Indices. Now, consider some periodic orbit with Mallet-Yorke Index $0$ and connect it with another period-doubling bifurcation. This creates a colored graph as in Figure \ref{nonadmissible}, with two edges, $A_1, A_2$, each corresponding to Mallet-Yorke Index $1$, connected by an edge corresponding to Mallet-Yorke Index $0$ (see the left image in Figure \ref{nonadmissible}). Now, consider, say, $A_2$, and perform a saddle-node bifurcation at both its ends, which allows the creation of a cycle, as on the right image in Figure \ref{nonadmissible}. This new colored graph lies in $G_{1,2}$, and yet, by Proposition \ref{rules2} we know it correspond to any $C^1$ one-parameter family $f_t:S^2\to S^2$ in $F_1(S)$, where $t\in(0,1)$.

 \begin{figure}[h]
\centering
\begin{overpic}[width=0.25\textwidth]{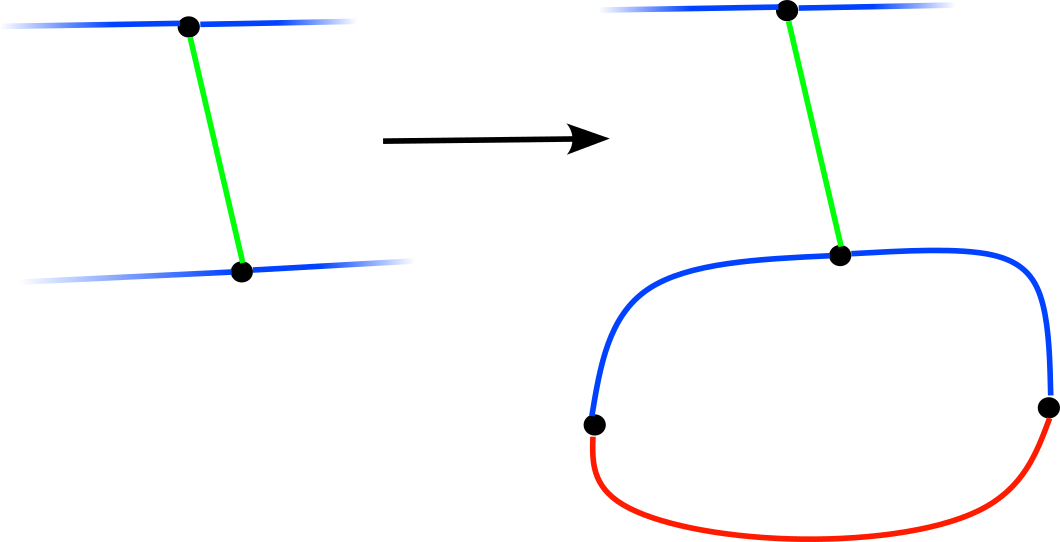}
\end{overpic}
\caption{\textit{The left graph consists of two edges denoting Mallet-Yorke Index $1$ (in blue) connected by a Mallet-Yorke Index $0$ edge (the green arc). Assuming $A_2$ is the lower arc, then in $G_{1,2}$ it can be closed by two saddle-node bifurcations connected by a $-1$ Mallet-Yorke Index edge (i.e., the red edge). The colored graph generated by this proccess (on the right) is admissible in $G_{1,2})$, but cannot be a bifurcation diagram in $F_1(S^2)$.}}
\label{nonadmissible}
\end{figure}

That being said, even though the argument above proves we cannot a priori assume $G_{k,d}\setminus B_k(S)=\emptyset$, it \textbf{does not} rule out $\Gamma\in G_{k,d}\setminus B_k(S)$ being a bifurcation subdiagram for some one-parameter family in $F_k(S)$ - or, put simply, even if $\Gamma$ cannot be realized as a bifurcation diagram on its own, it might be a subgraph of some other bifurcation diagram $\Gamma'\in B_k(S)$. In fact, we will prove in Appendix \ref{everygraph} that for every connected graph $\Gamma\in G_{k,d}$ there exists an isotopy $f_t:S\to S$ s.t. $\Gamma$ is a subgraph of its bifurcation diagram (although that isotopy may not necessarily lie in $F_k(S)$), see Theorem \ref{existence} for the proof. This discussion shows that even though $B_k(S)$ and $G_{k,d}$ cannot be assumed to coincide, intuitively $G_{k,d}$ can be thought of as an idealized moduli space of bifurcation diagrams which generalizes $F_k(S)$. Therefore, in the remainder of this paper we will analyze graphs inside the set $G_{k,d}$. As it will be clear, these graphs can be at least thought of as a very particular type of bifurcation diagrams, namely, routes to chaos.

\section{From bifurcation diagrams to graphs}
\label{tographs}
Having replaced bifurcation diagrams with colored graphs, we now proceed to analyze the collection of colored graphs $G_{k,d}$, using Graph Theory. To do this, we first recall several definitions and facts. Incorporating previous notations, a graph is always defined as a pair $G = (V, E)$ - where $V$ is the set of vertices or nodes and $E$ is the set of edges. In particular, we think of every element in $E$ as a set of unordered pairs $e=\{v_1, v_2\}$ of vertices in $V$, denoting which vertices on $G$ the edge $e$ connects. We will need the following definitions:

\begin{figure}[h]
\centering
\begin{overpic}[width=0.5\textwidth]{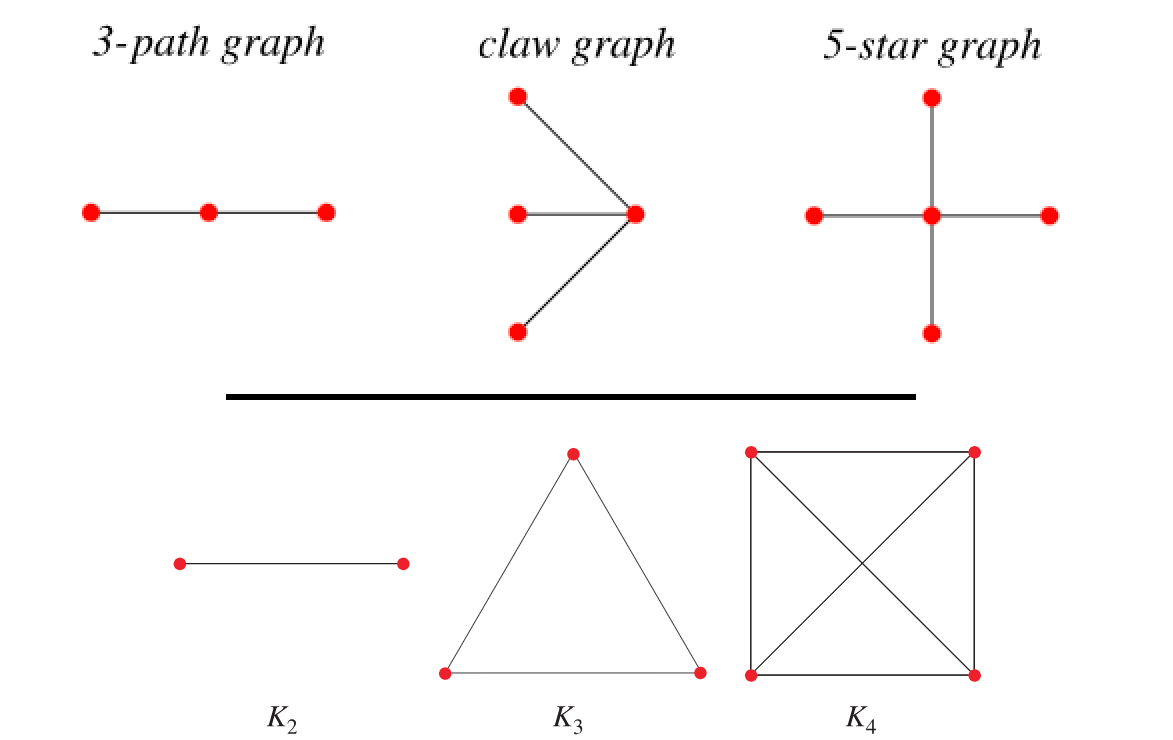}
\end{overpic}
\caption{\textit{Star graphs (on the upper level) and complete graphs (on the lower level). }}
\label{fig6}
\end{figure}

\begin{definition}
    \label{starcomplete} A graph $G=(V,E)$, where $|V|=n$, is said to be a \textbf{complete graph on $n$ vertices}, denoted by $K_n$, if each pair of vertices in $V$ is connected by an edge, $K_n$ (see the illustration in Figure \ref{fig6}). Conversely, $G$ is an \textbf{$n$--star graph} if it is a tree on $n$ vertices, with one vertex having degree $n-1$ and the other $n-1$ vertices having degree $1$ (see the illustration in Figure \ref{fig6}). 
\end{definition}
\begin{remark}
    An $n$-star is the complete bipartite graph $K_{(1,n-1)}$, whose vertices $V$ can be decomposed to a singleton $\{v_1\}$ and $\{v_2,...,v_n\}$ s.t. no edge in $E$ connects two vertices in $\{v_2,...,v_n\}$, while $v_1$ is connected by an edge to $v_j$, $2\leq j\leq n$ (see the illustration in Figure \ref{fig6}).
\end{remark}

The main obstacle we will face is that the sets $\{G_{k,d}\}$ are not collections of graphs - but rather a collection of colored graphs, where the coloring is defined by the Mallet-Yorke Index in accordance with Theorem \ref{indexinvariance} and Propositions \ref{rules2}-\ref{rules4}. As such, in order to study them we first need to find a way which on the one hand, allows us to reduce these complex collections of colored graphs to more manageable graphs and on the other hand, does not lose too much of the dynamical information in the process. As Figure \ref{fig6} suggests, we will now introducea way to do so - the \textbf{star representation}.

\begin{figure}[h]
\centering
\begin{overpic}[width=0.6\textwidth]{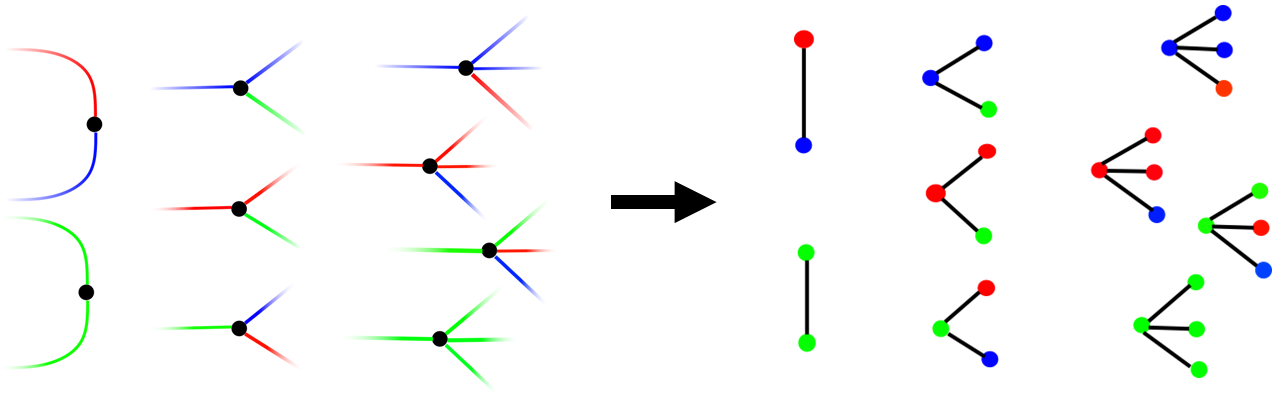}
\end{overpic}
\caption{\textit{The star representation of saddle-node, period-doubling and type $m\geq3$ bifurcations in d=4, see also Figure \ref{fig5}.}}
\label{project}
\end{figure}

The idea is the following: by Theorem \ref{indexinvariance} and Propositions \ref{rules2}-\ref{rules4} we know that every $\Gamma\in G_{k,d}$, $k>0,d>1$ can be decomposed to a collection of colored edges (where the color corresponds to the Mallet-Yorke Index), connected to one another by edges corresponding to bifurcation orbits. Therefore, if we change the role of edges and vertices we get a graph, with colored vertices instead of edges - which we can decompose into a collection of star graphs connected to one another at the vertices (see the illustration in Figure \ref{project} and Figure \ref{glue}). In detail, we transform the colored graphs in $G_{k,d}$ via star graphs as follows (see  Figures \ref{project}, \ref{project1}, \ref{glue}):

\begin{itemize}
    \item By definition, a saddle-node bifurcation is composed of two edges, $E_1,E_2$ colored with colors $c_1, c_2$ which are glued together at $v$. We map this to $2$-star graph with vertices colored by $c_1$ and $c_2$ corresponding to the Mallet-Yorke Index, as in Figure \ref{project}.
    \item Every period-doubling bifurcation is projected a $4$-star graph, where the edges $E_1,E_2$ and $E_3$ are colored with colors $c_1,c_2$ and $c_3$ (respectively), corresponding to the Mallet-Yorke Index. Assuming that the periodic orbit undergoes bifurcation corresponds to $E_1$, we map this to $3$-star graph with vertices colored by $c_1,c_2$ and $c_3$ s.t. $c_1$ has degree $2$, and $c_2$ and $c_3$ have degree $1$, see Figure \ref{project}.
    \item Similarly, a type $m\geq3$ bifurcation is defined as a $5$-star graph with edges $E_1,E_2,E_3$ and $E_4$, colored with colors $c_1,c_2,c_3$ and $c_4$ (respectively), again, corresponding to the Mallet-Yorke Index. Assume $E_1$ corresponds to the period orbit whose period is multiplied by $m$ - similarly to the period-doubling case, we map this to $4$-star graph, where $c_1$ becomes a vertex of degree $3$ connected to three vertices $c_2,c_3$ and $c_4$, see Figure \ref{project}. 
    \item Finally, recall an $n$-junction, $n>3$, is a periodic orbit that splits into $n$ distinct periodic orbits. This again defines an  $n+2$-star graph on edges $E_1,...,E_{n+2}$ which are colored with colors $c_1,...,c_{n+2}$  (respectively), corresponding to the Mallet-Yorke Index. Assuming again that $E_1$ corresponds to the bifurcating orbit, we similarly map this to  $n+1$-star graph with colored vertices s.t. $c_1$ becomes a vertex of degree $n+1$, see Figure \ref{project1}.
    \item The representation of a diagram from $\Gamma\in G_k(S)$ is defined as the representation of its decomposition to these star graphs, glued at the corresponding vertices.
\end{itemize}

\begin{figure}[h]
\centering
\begin{overpic}[width=0.5\textwidth]{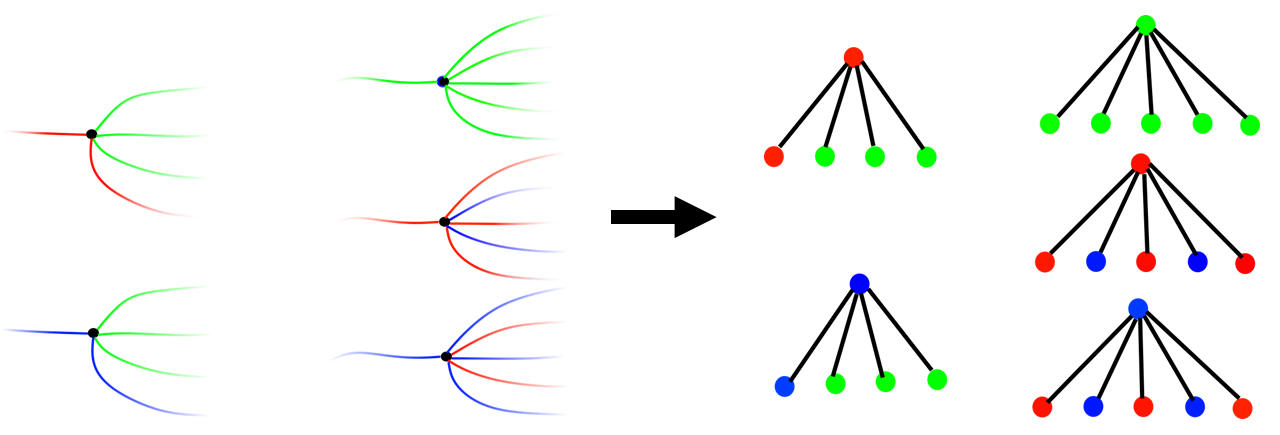}
\end{overpic}
\caption{\textit{The star representation of all possible $4$- and $5$-junctions in $d=4$, see also Figure \ref{rules4d}.}}
\label{project1}
\end{figure}

As stated above, we will refer the above projection of $G_{k,d}$ into the collection of colored graphs as the \textbf{star representation}. Before moving on, we remark this projection can be intuitively described as "interchanging colors and vertices". As will be made clear soon, the reason we do so is that by encoding the colors in vertices we can count bifurcation diagrams more easily - at least in some case. As we will see, this approach will allow us to answer negatively the question "can we characterize the general route to chaos" (Theorem \ref{treeth}). Later on, in Appendix \ref{block}, for the sake of completeness, we will discuss another projection: the \textbf{clique representation}.

\subsection{Trees}
\label{trees}
A \textbf{tree} is a graph in which every pair of distinct vertices is connected by exactly one path - or equivalently, a connected acyclic graph. A \textbf{forest} is a graph s.t. each component of it is a tree. The \textbf{degree of a tree} refers to the highest number of children, any node in the tree has:  a binary tree is degree $2$, a ternary tree is degree $3$ and an $m$-ary tree, $m>0$ is one where each node has no more than $m$ children. In this Subsection we will use trees to answer the following question - can we find an upper bound on the number of possible routes to chaos?

To motivate our approach, let $S$ be a closed, $d>1$-dimensional smooth manifold. By \cite{KY4}, given a $d$-dimensional cube $ABCD$ and a $C^1$ one-parameter family $f_t:ABCD\to\mathbb{R}^d$, $t\in[0,1]$ s.t. $f_0(ABCD)\cap ABCD=\emptyset$ and $f_1(ABCD)$ is a Horseshoe map, then, under certain natural constrains, all the periodic orbits in $ABCD$ arise through period-doubling cascades as in Figure \ref{cascades} (for the precise conditions, see Section 2 in \cite{KY4}). In this setting  every cascade corresponds to a colored tree while the bifurcation diagram of the entire one-parameter family -- to a colored forest. This example illustrates that one can think of colored trees with infinitely many vertices and colored forests with infinitely many components in $G_{k,d}$ as representing bifurcation diagrams in which "dynamical complexity can only be added". More precisely, following \cite{EY}, we say a $C^1$ one-parameter family $f_t:S\to S$, $t\in(0,1)$ in $F_k(S)$, $k\geq1$, satisfies the \textbf{monotonicity property in $A\times(0,1)$}, where $A$ is some open set in $S$, if, once a periodic of orbit arises in $A$ at some $t_0\in(0,1)$, it persists as a periodic orbit for $f_t$ in $A$ for all $t>t_0$. When this is the case, it is easy to see the subgraph of the bifurcation diagram $\Gamma\in B_k(S)$, describing the evolution of periodic dynamics in $A$, corresponds to some tree.

All in all, this discussion motivates us to think of trees as the part of the bifurcation diagram describing possible route to chaos. And indeed, as we show in the proof of Theorem \ref{existence}, given any forest which is a countable collection of trees $\Gamma\in G_{k,d}$ there exists an isotopy $f_t:S\to S$ of continuous maps with periodic orbits, together with an open set $A\subseteq S$ s.t. the following holds:
\begin{itemize}
    \item If $\{\Gamma_i\}$ are the trees composing the forest $\Gamma$, there exist open neighborhoods $\{A_i\}$ and periodic orbits $x_i\in A_i$ bifurcating as dictated by $\Gamma_i$ as we vary $t$ to $1$.
    \item As $t$ is varied towards $1$, $x_i$ bifurcates in a way that all the periodic orbits, arising from it, persist. In other words, even if there are other dynamics occurring in $A_i$, there is one collection of periodic orbits where the complexity can only increase as $t\to1$ (i.e., those created by bifurcating from $x_i$).
\end{itemize}

Before we proceed (and to better motivate the proof below), we recall a spanning tree of a connected graph is a subgraph which includes all the original graph’s vertices, and has no cycles. Moreover, if the collection of spanning trees includes all trees even if they are isomorphic as graphs but having different labels, then a graph isomorphism will induce an isomorphism on spanning trees and vice-versa. Hence, a graph may have exponentially many spanning trees. This shows that even in a graph-related context, trees are the most important object.

\begin{figure}[h]
\centering
\begin{overpic}[width=0.42\textwidth]{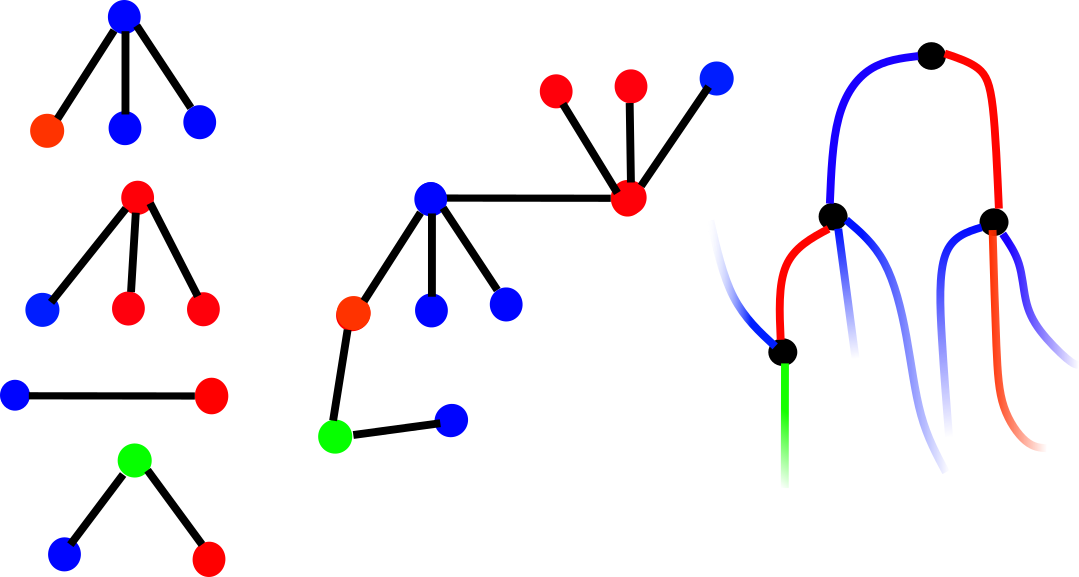}
\end{overpic}
\caption{\textit{On the left - a collection of (colored) star graphs. In the middle - a graph constructed by gluing vertices with the same color. On the right - the corresponding bifurcation diagram, which is admissible when $d>2$, see Proposition \ref{rules3}.}}
\label{glue}
\end{figure}

Let $G^A_{k,d}$ denote the collection of all possible connected acyclic diagrams from $G_{k,d}$ - i.e., the collection of all trees in $G_{k,d}$. From the above, if we consider $G_{k,d}$ as a "moduli space" of all possible admissible bifurcation diagrams in $d$-dimensional manifolds, $G^A_{k,d}$ intuitively corresponds to all possible routes to chaos that arise via admissible bifurcations, i.e., saddle-node, period-doubling, type $m\geq3$ and $j$-junctions, $k+1\geq j\geq 4$. To continue, for any natural number $n$ denote by $P^nG^A_{k,d}$ the set consisting of all possible connected acyclic diagrams of degree $k+1$ on $n$ vertices in the star representation of $G^A_{k,d}$. Finally, given $k>r>0$, by the \textbf{relative dimension of $G^A_{k,d}$ in $G^A_{r,d}$}, we will always mean the limit $\frac{|G^A_{k,d}|}{|G^A_{r,d}|} = \lim\limits_{n \rightarrow \infty} \frac{|P^nG^A_{k,d}|}{|P^nG^A_{r,d}|}$. In Theorem \ref{treeth} we will prove that for all $r>k$ and every dimension $d\geq4$, we always have $\frac{|G^A_{r,d}|}{|G^A_{k,d}|}=\infty$. But for that we first prove the following Lemma:

\begin{lemma}
\label{lemma1}If we remove the colors from the vertices of the trees in $P^nG^A_{k,d}$ with $d\geq4$, the resulting collection of graphs is all trees of degree $k+1$ on $n$ vertices.
\end{lemma}
\begin{proof}
By Proposition \ref{rules4}, it would suffice to prove the assertion for $d=4$ - since by the said Proposition, for all $d>4$, $G^A_{k,d}=G^A_{k,4}$. As discussed earlier, when $d=4$ we can bifurcate a given periodic orbit in almost every way w.r.t. the Mallet-Yorke Index. In detail, taking the Mallet-Yorke Indices $0,1,-1$ as colors, we can split the periodic orbit in a bifurcation in almost every pattern (see Figure \ref{fig5}). By definition, any tree is the combination of star graphs via the following relation: a vertex of degree $i>1$ corresponds to $i$-star graph. It is therefore easy to see that due to Proposition \ref{rules4} we can always connect any admissible saddle-node, period-doubling or type $m \geq 3$ bifurcation diagram at any vertex with a bifurcation diagram of any needed degree. When we remove the colors, each of these bifurcations corresponds to a star graph. In the specific case of $G^A_{k,d}$ and the vertices of the colored graphs in it, this proves that when we remove the colors from the vertices of trees in $P^nG^A_{k,d}$ we get all tree of degree $k+1$ on $n$ vertices.
\end{proof}

Having proven Lemma \ref{lemma1}, we now use it to prove Theorem \ref{treeth} where we show that whenever $d\geq4$, the relative dimension of $G^A_{r,d}$ in $G^A_{k,d}$ is infinite. In other words, there is probably no upper bound for the possible complexity of chaos in high dimensions.

\begin{theorem} \label{treeth} For any natural $k>1$ and $d\geq4$, we have $\frac{|G^A_{k+1, d}|}{|G^A_{k,d}|} = \frac{|G^A_{k, d}|}{|G^A_{2,d}|} = \frac{|G^A_{k,d}|}{|G^A_{1,d}|} = \infty$.
\end{theorem}
\begin{proof}
To begin, we first recall the formula for the number of $k$-ary (uncolored) trees with $n$ vertices \cite{tree}:
\begin{equation*}
    \frac{1}{(k-1)n+1} \binom{kn}{n} =\frac{1}{(k-1)n+1} \frac{(kn)!}{((k-1)n)!n!}.
\end{equation*}

By Lemma \ref{lemma1} we know that $P^nG^A_{k-1, d}$, after forgetting colors, is the set of all trees of degree $k$ on $n$ vertices - therefore, since every vertex is colored by one of three colors we conclude $|P^nG^A_{k-1,d}|\leq 3^n N(k)$, where $N(k)$ is the number of $k$-ary trees with $n$ vertices. This yields the following inequality:

\begin{equation*}
    \frac{|P^nG^A_{k,d}|}{|P^nG^A_{k-1,d}|}\geq\frac{N(k+1)}{3^nN(k)}.
\end{equation*}
By using the Stirling approximation (i.e. $n! \sim n^{n+1/2} e^{-n}$; $(1 \pm 1/a)^b \sim e^{\pm b/a}$), we now get:

\begin{equation*}
    \frac{N(k+1)}{3^nN(k)}\approx \frac{(k+1)^{(k+1)n} (k-1)^{(k-1)n}}{3^nk^{2kn}}=\frac{1}{3^n} (1-\frac{1}{k^2})^{(k-1)n} (\frac{k+1}{k})^{2n}.
\end{equation*}
We now note that by $k>1$ we have $(1-\frac{1}{k^2})^{(k-1)n}>\frac{1}{2^{(k-1)n}}$, therefore we conclude:

\begin{equation*}
    \frac{N(k+1)}{3^nN(k)}\approx \frac{1}{3^n} (1-\frac{1}{k^2})^{(k-1)n} (\frac{k+1}{k})^{2n}\geq\frac{1}{6^{(k-1)n}}(\frac{k+1}{k})^{2n}.
\end{equation*}

Write $c=6^{k-1}$ and note $c^n = (\frac{k+1}{k})^{n\log_c(\frac{k+1}{k})}$. Consequently, $\lim\limits_{n\to\infty}(\frac{k+1}{k})^{2n}c^{-n}=\infty$, which implies $$\frac{|G^A_{k,d}|}{|G^A_{k-1,d}|}=\lim_{n\to\infty}  \frac{|P^nG^A_{k,d}|}{|P^nG^A_{k-1,d}|}=\infty.$$
\end{proof}
\begin{remark}
    One particular implication of Theorem \ref{treeth} is that $\frac{|G^A_{3,d}|}{|G^A_{2,d}|}=\infty$ for all $d\geq4$ - i.e., intuitively, there are much more routes to chaos involving type $m\geq3$ bifurcations than there are only period-doubling routes to chaos.
\end{remark}
Having studied the case where the dimension $d\geq4$, we are now led to ask: do these results also hold for lower-dimensional systems? In other words, assuming $d=2,3$, given $k>r\geq1$ do we still have $\frac{|G^A_{k},d|}{|G^A_{r,d}|}=\infty$? Unfortunately, when $d<4$ there is no clear analogue of Lemma \ref{lemma1}, so one cannot perform approximation as above. To better illustrate this, we note, based on Propositions \ref{rules2}, \ref{rules3}, that the following occurs:
\begin{itemize}
    \item $\cup_{k, n>0} P^nG^A_{k,2}$ does not include all binary and ternary trees (after forgetting the colors).
    \item $\cup_{k, n>0}P^n G^A_{k,3}$ includes all binary trees, but not all $4$-ary trees (again, after forgetting the the colors).
\end{itemize}

That being said, surprisingly, when $d=2,3$ the relative dimensions are dominated, in a sense, by higher-dimensional bifurcation diagrams. To make this statement precise, we make the following computation:

\begin{equation*}
    \frac{|P^nG^A_{k+1,d}|}{|P^nG^A_{k,d}|} = \frac{|P^nG^A_{k+1,d}|}{|P^nG^A_{k,d}|} \frac{|P^nG^A_{k+1, d+1}|}{|P^nG^A_{k+1,d+1}|} = \frac{|P^nG^A_{k+1,d}|}{|P^nG^A_{k+1,d+1}|} \frac{|P^nG^A_{k+1, d+1}|}{|P^nG^A_{k,d}|}.
\end{equation*}

By Proposition \ref{rules2} and Proposition \ref{rules3} we have the following two inequalities:

\begin{equation*}
    \frac{|P^nG^A_{k+1, d+1}|}{|P^nG^A_{k,d}|} \geq \frac{|P^nG^A_{k+1, d+1}|}{|P^nG^A_{k,d+1}|} \geq 1 \text{ and }
    0<\frac{|P^nG^A_{k+1,d}|}{|P^nG^A_{k+1,d+1}|}\leq 1.
\end{equation*}
When $d=3$, we already know $\lim\limits_{n\to\infty} \frac{|P^nG^A_{k+1, d+1}|}{|P^nG^A_{k, d+1}|}=\infty$, which shows that the quantity $\frac{|G^A_{k+1,d}|}{|G^A_{k,d}|}=\lim\limits_{n\to\infty} \frac{|P^nG^A_{k+1,d}|}{|P^nG^A_{k,d}|}$ depends on $\lim\limits_{n\to\infty}\frac{|P^nG^A_{k+1,d}|}{|P^nG^A_{k+1, d+1}|}$. From now on, for $d=2,3$ and $k>0$ we denote $\lim\limits_{n\to\infty}\frac{|P^nG^A_{k+1,d}|}{|P^nG^A_{k+1,d+1}|}=\frac{|G^A_{k,d}|}{|G^A_{k,d+1}|}$, and refer to this as the \textbf{relative share of $G^A_{k, d}$ inside $G^A_{k,d+1}$} - by definition, $\frac{|G^A_{k,d}|}{|G^A_{k,d+1}|}\leq1$. Much like how the relative dimension measures the growth rate of the different possibilities for different routes to chaos, the relative share measures how much of the routes to chaos in dimension $d+1$ are essentially $d$-dimensional. Recalling $\frac{|G^A_{k+1,4}|}{|G^A_{k,4}|}=\infty$ for all $k>0$ (see Theorem \ref{treeth}), these calculations imply the following:

\begin{corollary}
\label{dim3} Assume $d=2,3$. Given any natural number $k>0$, if the relative dimension $\frac{|G^A_{k+1,d}|}{|G^A_{k,d}|}$ is finite, then the relative share $\frac{|G^A_{k+1,d}|}{|G^A_{k+1,d+1}|}$ is also finite.
\end{corollary}
In other words, the growth rates of routes to chaos in dimension $3$ depends on how many bifurcation scenarios in dimension $4$ are essentially three-dimensional, i.e., can be realized along a three-dimensional $C^1$ one-parameter family of diffeomorphisms. At this point we remark that something similar happens in dimension $2$ with respect to dimension $1$ - but to introduce it, we first have to extend our setting. To do so, recall the Mallet-Yorke Index can be extended to periodic orbits in dimension $1$ (see, for example, \cite{EY}). However, as smooth one-dimensional isotopies $f_t:\mathbb{R}\to \mathbb{R}$, where $t\in[0,1]$, cannot have routes to chaos (or non-trivial dynamics, in general), in order to get meaningful results one has to replace these isotopies with $C^1$-homotopies.

\begin{figure}[h]
\centering
\begin{overpic}[width=0.42\textwidth]{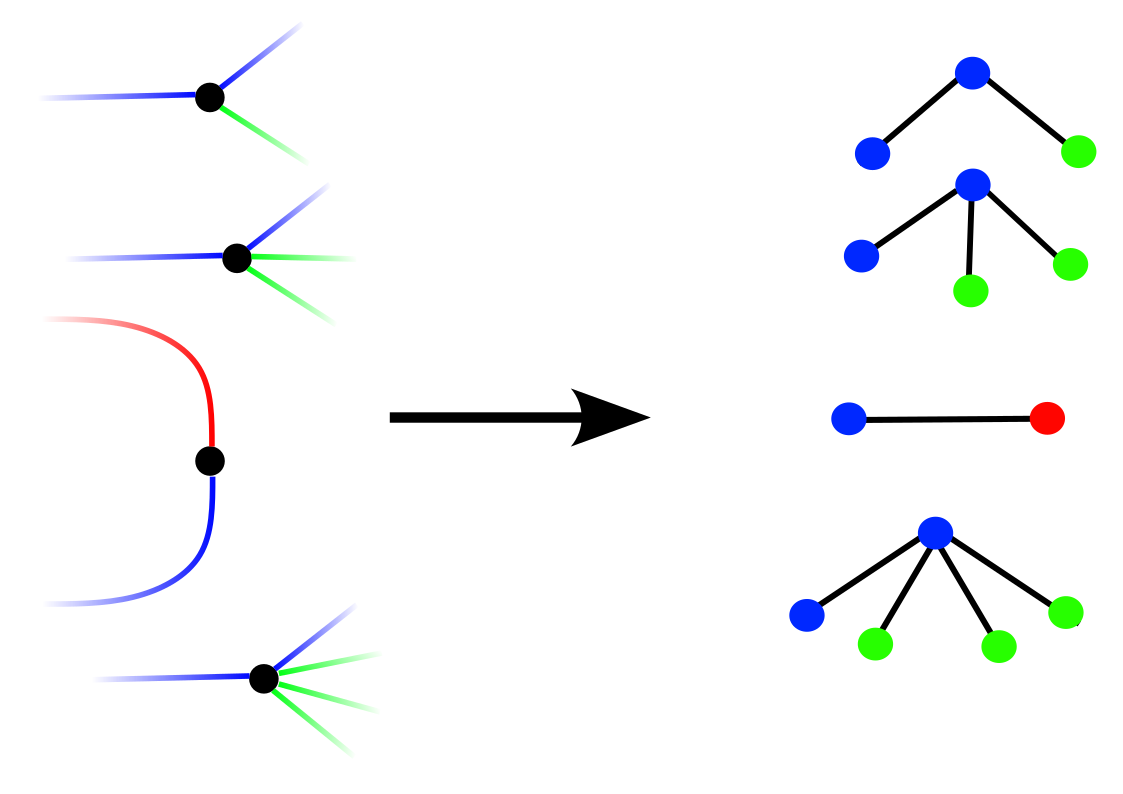}
\end{overpic}
\caption{\textit{On the left - admissible bifurcations. On the right - their star representations.}}
\label{onedim}
\end{figure}

In this new setting, one can define the Mallet-Yorke Index of a periodic point $y$ in a similar way as in higher dimensions. However, in this setting the only admissible bifurcations are saddle-node, period-doubling, and $n$-junctions modeled on them (see the illustration in Figure \ref{onedim}). This implies that given any $k>0$, by analogously defining $G^A_{k,1}$, we always have $G^A_{k,1}\subseteq G^A_{k,2}$. Or, in other words, every one-dimensional bifurcation diagram is also an admissible two-dimensional diagram. In particular, as there are no type $m$ bifurcations in $G^A_{1,2}$, we have $G^A_{1,1}=G^A_{1,2}$.

Heuristically speaking, the dynamics of two-dimensional maps are essentially those of "stretched out" one-dimensional maps (see, for example, \cite{Bo} for a survey of these ideas). Moreover, it is also known that four-dimensional flows can have three-dimensional subsystems - which in turn can have two-dimensional discrete time subsystems with chaotic behavior (with the restricted three body problem being one such famous example - see \cite{LM}). We believe the relative dimension and relative share possibly explain why we encounter such phenomena in two- and three-dimensional systems. Motivated by these facts, we ask the following: 

\begin{Question}
    \label{dimconj1} For all $k>1$, is the relative share $\frac{|G^A_{k,1}|}{|G^A_{k,2}|}$ necessarily finite? And, conversely, can the relative shares $\frac{|G^A_{k,2}|}{|G^A_{k,3}|}$ and $\frac{|G^A_{k,3}|}{|G^A_{k,4}|}$ be non-zero? 
    \end{Question}
At present, we can only refer to Section 5 in \cite{san2}, where routes to chaos excluding period-doubling cascades were conjectured to be atypical in dimensions $2$ and up (see \cite{san2} for the precise formulation). If that conjecture holds, then its proof could possibly shed some light on the answer to the question above. Another possible related fact is Theorem 2.7.II from \cite{AB}, which proves certain types of period-doubling cascades can occur only for certain class of four-dimensional flows, but not for three-dimensional flows. That being said, despite these results and conjectures, at present we do not know which answer should be expected - as stated at the Introduction, some numerical evidence from \cite{Gallas} and \cite{Tur} as well as the theoretical evidence from \cite{AS} and \cite{GuW} could equally be interpreted as evidence in the opposite direction.

Continuing this line of thinking, before concluding this Subsection, we recall that it is possible to convert any uncolored $m$-ary tree to an uncolored binary tree - which, intuitively, is like saying all trees in $G^A_{k,d}$, $k>1$ can be converted to trees in $G^A_{1,d}$. In detail, recall that to form a binary tree from an arbitrary $m$-ary tree one needs to do the following: 
\begin{enumerate}
    \item The root of the tree is made the root of the binary tree.
   \item Starting with the root, each node’s leftmost child in the tree is made its left child in the binary tree. 
    \item Any subsequent children of that node in the tree are treated as siblings. The second child becomes the right child of the first child, the third child becomes the right child of the second child, and so on.
\end{enumerate}

This combinatorial argument tells us that even if for all $d>0$ we have $\frac{|G^A_{k,d}|}{|G^A_{1,d}|}=\infty$, we should still expect period-doubling routes to chaos to play a prominent role. This should be contrasted with Section 6 in \cite{EY} and  Corollary 6.5, Theorems 1, 2 in \cite{san2}, which, briefly speaking, state that period-doubling routes to chaos are the "general case" for large classes of high-dimensional one-parameter families. Assuming this to be the general case, these results could possibly be interpreted in light of the graph-theoretic argument above, namely, that routes to chaos involving periodic orbits are essentially binary trees with some "extra" information. That being said, by Theorem \ref{treeth}, that "surplus" information could potentially include a lot more extra data.  

\section{Discussion}

Before we conclude this paper, we would like to say a few words on how this work can possibly be continued. Despite Theorem \ref{treeth} being applicable only for trees, we believe one can use Graph Theory to describe "universal" properties of bifurcation diagrams. That being said, provided these "universal" properties are shown to exist, they can probably shed light on the universal mechanisms governing the evolution from order into chaos.

Let us recall the \textbf{Rado graph} (the said graph arises via the Erdős–Rényi model of a random graph on countably many vertices \cite{erdos}), which can be characterized in a deterministic way by the next properties \cite{rado}:
\begin{itemize}
    \item \textbf{Universality} - it has countably many vertices, and any graph with countably many vertices is isomorphic to an induced subgraph of the Rado graph.
    \item \textbf{Homogeneity} - any isomorphism between finite induced subgraphs can be extended to an isomorphism between the Rado graph to itself.
    \item \textbf{Robustness} - removing any finite set of its vertices and edges produces a graph isomorphic to the whole Rado graph
\end{itemize}

The above properties define the Rado graph uniquely up to isomorphism. The universality of the Rado graph can be extended to edge-colored graphs - or in our notation, to some conjectured universal "bifurcation diagram", realizing all possible bifurcation diagrams. Moreover, the first-order logic sentences ($0-1$ laws) that are true for the Rado graph are also true of almost all random finite graphs - while the sentences that are false for the Rado graph are also false for almost all finite graphs (for the precise details, see \cite{logic}). 

We conjecture that something similar happens in Dynamical Systems. To elaborate, we recall the \textbf{Theorem of }\textbf{Multiversal chaos}, proven in \cite{Den2}. That theorem states the existence of a dynamical system $R:X\to X$ (where $R$ is a renormalization operator - see \cite{Den2}) s.t. the following holds:
\begin{itemize}
    \item The periodic orbits for $R$ are dense in $X$.
    \item if $g:S\to S$ is a dynamical system on some finite-dimensional set $S$ (deterministic or stochastic), then there exist infinitely many invariant sets $X'\subseteq X$ s.t. $R|_{X'}$ is conjugate to $g:S\to S$
\end{itemize}

In light of this result, we conjecture a bifurcation analogue of the Theorem of Multiversal chaos - namely, that much like all dynamical systems are governed by a fixed rule (i.e., $R:X\to X$), so do all bifurcations. Specifically, we conjecture the following:

\begin{Conjecture}
\label{univer1}    There exists a homotopy $R_t:X\to X$, $t\in[0,1]$ s.t. the following is true:
    \begin{itemize}
        \item $R_1=R$, while $R_0$ has "simple" dynamics (for example, a finite number of periodic orbits). In particular, the bifurcation diagram $\frak{R}$ of this curve is well-defined.
        \item Let $g_t:S\to S$, $t\in[0,1]$, be a $C^1$-isotopy of some closed, smooth manifold $S$ of dimension $d>1$, and let $\Gamma_g$ denote the bifurcation diagram of $g_t$. Then, $\Gamma_g$ can be embedded as a subgraph of $\frak{R}$.
    \end{itemize}
    \end{Conjecture}

Provided Conjecture \ref{univer1} can be proven, its implications would be that there exists a universal theory of bifurcations for $C^1$ one-parameter families. More specifically, it would show that in order to study all possible bifurcations, one has to study the graph $\frak{R}$ (although in many ways, this could prove much more difficult then studying separately some given $C^1$-isotopy $g_t:S\to S$, $t\in[0,1]$).

Before we conclude this discussion, we remark that we have not touched upon the question of whether our results can be used to study flows generated by differential equations. As our results are strongly dependent on the notion of the Mallet-Yorke Index which was originally introduced to study persistent periodicity for flows in \cite{PY}. As such, intuitively, one would expect the answer to be true. Moreover, our results easily generalize to $C^1$-curves of suspension flows generated by suspending $C^1$-isotopies $g_t:S\to S$, $t\in[0,1]$. However, when one is given a $C^1$-curve of vector fields, $\dot{s}=F_t(S)$, $t\in[0,1]$, $s\in M$ (where $M$ is some closed smooth manifold of dimension at least $3$), the problem of describing the bifurcation diagram becomes much harder - mostly due to the involvement of fixed points, which allow bifurcations not happening in $F_k(S)$, $k>0$ - see, for example, \cite{PY2} for a survey of these ideas.

Finally, we would like to state we believe there are additional ways to apply Graph Theory to study genericity in bifurcation diagrams. For example, consider the notion for the limit of a sequence of graphs, also known as \textbf{graphons} (see \cite{lov}). It tells us that graphs themselves can be treated as measurable functions on $[0, 1]^2$, essentially by drawing their adjacency matrices. This allows for the treatment of graphs and random graphs as the same kind of object and analyze the limiting behavior of the sequence by considering the limiting behavior of the functions. We think the study of such objects can lead to new results about relative dimensions and generic bifurcation phenomena.

\section{Appendix - Block graphs}
\label{block}

Similarly to the star representation described at the beginning of Section \ref{tographs}, in this Appendix we also define and discuss analogous projection of $G_{k,d}$ using complete graphs, the \textbf{clique representation}. We define this as follows:

\begin{itemize}
\item Given "building blocks" of graphs, like saddle-node, period-doubling, type $m \geq 3$ bifurcations and $n\geq4$-junction, we interchange colors with vertices exactly as before - yet, this time we represent each such "building block" as complete graphs, see Figure \ref{completed}. Note here that $2$-star graph is a complete graph, and that in both cases it is the projected image of a saddle-node bifurcation orbit.
\item Given $\Gamma\in G_k(S)$, we project the collection of "bifurcation laws" defining it to the complete graphs, which we then glue at the corresponding vertices, see Figure \ref{completed}.
\end{itemize}

\begin{figure}[h]
\centering
\begin{overpic}[width=0.4\textwidth]{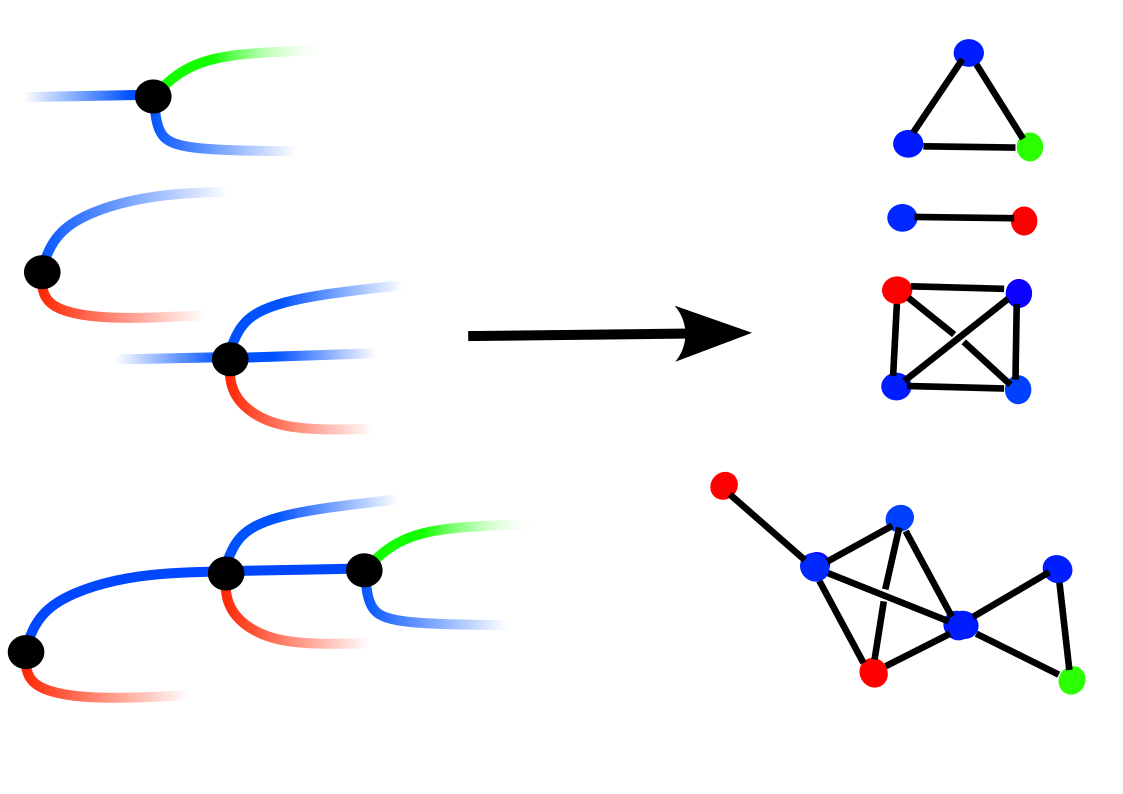}
\end{overpic}
\caption{\textit{The clique representation - the upper three "building blocks" are projected to complete graphs, while the lower bifurcation diagram is projected to the gluing of these building blocks at the corresponding vertices. }}
\label{completed}
\end{figure}

By definition, the last never gives trees and forests, but let us again notice that the 2-star graph is the complete graph on two vertices. Here we will show that two representations are related and dual in some sense (see Theorem \ref{conjugate}). Intuitively, this says that the properties of bifurcation diagrams that are trees are analogous to more complex bifurcation structures. 

\begin{figure}[h]
\centering
\begin{overpic}[width=0.35\textwidth]{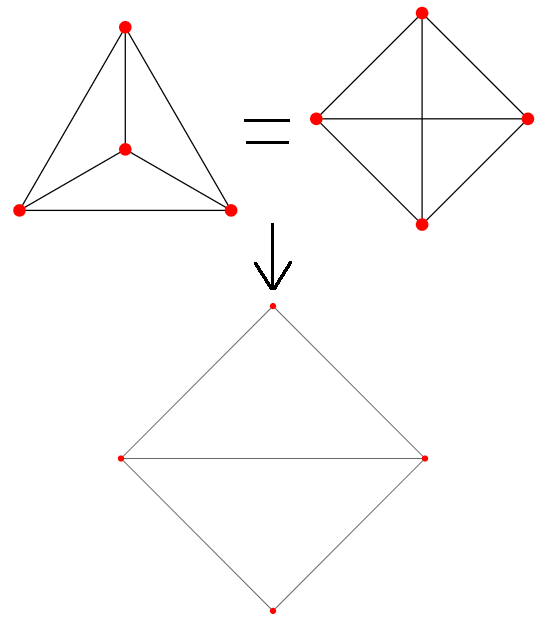}
\end{overpic}
\caption{\textit{The transformation from $K_4$ to the Diamond graph. }}
\label{fig7}
\end{figure}

We first recall several notions. Any connected graph decomposes into a tree of biconnected components (that is, if any one vertex were to be removed, the graph would remain connected). We refer to such components as the \textbf{block-cut tree} of the graph. These block-cut trees are attached to each other at shared vertices called \textbf{cut vertices}, i.e. a cut vertex is any vertex whose removal increases the number of connected components. Finally, a \textbf{block graph} or \textbf{clique tree} is a type of connected graph in which every biconnected component is a clique (that is, a complete graph - see Definition \ref{starcomplete}). Note that block graphs can also be described via a forbidden graph characterization, i.e., as the graphs that do not have the following components:

\begin{itemize}
    \item The diamond graph, that is, the complete graph $K_4$  minus one edge (see the illustration in Figure \ref{fig7}).
    \item A cycle of four or more vertices as an induced subgraph, which is formed from a subset of the vertices of the graph and all of the edges from the original graph, connecting pairs of vertices in the said subset.
\end{itemize}

It is easy to see any tree satisfies the condition above - yet it allows a much more general structure. As with trees, the numbers of connected block graphs on $1, 2, \dots$ nodes are known, and are given in \cite{PL}:

\begin{proposition}
Let $(n_2, n_3, \dots)$ be a sequence of non-negative
integers and $n = \sum\limits_{i \geq 2} n_i(i-1) + 1$. Then the number $Hu(n_2, n_3, \dots)$ of block graphs on $n$ vertices having $n_i$ blocks of size i for each $i$, is given by
$$Hu(n_2, n_3, \dots) = \frac{(n-1)!}{\prod\limits_{j \geq 1} (j!)^{n_{j+1}} n_{j+1}!} n^{k-1},$$
where $k= \sum\limits_{j \geq 2} n_j$ is the total number of blocks.
\end{proposition}

We further recall that the \textbf{line graph} of a graph $G$ is another graph $L(G)$ that represents the adjacencies between edges of $G$. In detail, the line graph $L(G)$ is constructed in the following way: 
\begin{itemize}
    \item For each edge in $G$, assign a vertex in $L(G)$.
    \item  For every two edges in $G$ that have a vertex in common, there exists an edge connecting their corresponding vertices in $L(G)$. 
\end{itemize}

The most important fact for our purposes on line graphs is the following - for connected graphs with more than four vertices, there is a one-to-one correspondence between the isomorphisms of the graphs and the isomorphisms of their line graphs \cite{Whitney}. Using this fact, we now prove that representations using star graphs and complete graphs are dual in a sense, with which we conclude this Subsection:

\begin{theorem} \label{conjugate} Let $\Gamma$ be a bifurcation diagram which is a tree - then, the line graph of $\Gamma$ is a block graph. Consequently, everything we prove for bifurcation diagrams using the star representation is also true w.r.t. the clique representation.
\end{theorem}
\begin{proof}
The line graphs of trees are exactly the block graphs in which every cut vertex is incident to at most two blocks, or equivalently the claw-free block graphs (i.e., the graphs without an induced subgraph in the form of a three-leaf tree). For the proof, see \cite{Harary}, \cite{Beineke} and \cite{RW}. 
\end{proof}

\section{Appendix - Any graph from $G_{k,d}$ is a part of a bifurcation diagram}
\label{everygraph}

Let $S$ be a closed smooth manifold of dimension $d>1$ and choose some $k>0$. As shown in Section \ref{laws}, in general, the set $G_{k,d}$ is larger than $B_k(S)$, in the sense that the inclusion $B_k(S)\subseteq G_{k,d}$ should be expected to be strict (see the discussion immediately after Definition \ref{graphset}). Here we study the relationship between $G_{k,d}$ and $B_k(S)$ more closely. We begin by proving that despite $G_{k,d}$ being in general larger than $B_k(S)$, the set $G_{k,d}$ is not "too far away" from $B_k(S)$. In detail, we prove:

\begin{theorem}
    \label{existence} Let $\Gamma\in G_{k,d}$ be a connected colored graph (possibly with countably many vertices). Then, there exists a closed smooth manifold $S'$ of dimension $d'\geq d$ and an isotopy $f_t:S'\to S'$, $t\in[0,1]$ of continuous maps with periodic orbits, which bifurcate as dictated by $\Gamma$. More precisely, there exists a graph $\Gamma'$ encoding the bifurcations of some periodic orbits for the isotopy $f_t:S'\to S'$ s.t. $\Gamma\subseteq \Gamma'$.
\end{theorem}
\begin{proof}
  Assume $dim(S)=d>1$. We first prove the theorem under the assumption $\Gamma$ is a tree. To do so, begin by embedding $\Gamma$ as a curve in $\mathbb{R}^d\times[0,1]$. To continue, blow up the (possibly branched) curve $\Gamma$ into a precompact "branched tube" $T$ (not necessarily with the same topological type of $\Gamma$) - in particular, we choose $T$ satisfying the following:
  \begin{enumerate}
      \item $\overline{T}$ is compact in $\mathbb{R}^d\times[0,1]$.
      \item $\partial T$ is a smooth $n$-dimensional manifold.
  \end{enumerate}
  
Now, consider some orientation-preserving isotopy $g_t:T\cap \mathbb{R}^d\times\{t\}\to \mathbb{R}^n$, $t\in(0,1)$ with at least one fixed point in the interior of $T$, $x$, whose bifurcation diagram is $\Gamma$. Such an isotopy exists, since by the Lefschetz Theorem we can always split a fixed point while preserving the Lefschetz number. It is easy to see, the colorings for saddle-node, period-doubling, type $m\geq3$ and $n>3$-junctions preserve the Lefschetz number - therefore, as we can always split the said periodic orbit based on its Lefschetz number, we can ensure the said fixed point $x$ splits as $t$ is varied in a way that its bifurcation diagram is $\Gamma$ (in fact, the Mallet-Yorke Index can also be defined using the Lefschetz number - see \cite{PY3}). Moreover, we choose $g_t$, $t\in[0,1]$ s.t. $g_t$ can always be extended smoothly and diffeomorphically over $\partial T$. Or, in other words, we choose $g_t$ s.t. for all $t\in[0,1]$ it is smooth around $\partial T$ (but possibly not around the periodic orbits bifurcating from $x$).

As $T$ is an open set of $\mathbb{R}^d\times[0,1]$, we can choose $T$ s.t. for all $t\in[0,1]$ the sets $T\cap(\mathbb{R}^d\times\{t\})$ and $g_t(T)\cap(\mathbb{R}^d\times\{t\})$ lie in some fixed chart neighborhoods of $S$. Moreover, setting $g:T\times(0,1)\to T\times(0,1)$ as $g(s,t)=g_t(s)$, $g$ can be extended to a homeomorphism $F:S\times[0,1]\to S\times[0,1]$ due to Theorem $5.5$ in \cite{Pal} and we can also choose that extension to be a diffeomorphism in $(S\times[0,1])\setminus (T\times[0,1])$.

To continue, smoothly deform $F$ isotopically (if necessary) s.t. for all $t$, $F(S\times\{t\})=S\times\{t\}$. This can be done without changing the behavior of $F$ on $T$, as by the definition we have $g(s,t)=g_t(s)\in S\times\{t\}$. Consequently, we have an isotopy of continuous maps $f_t:S\to S$, $t\in[0,1]$ defined by $f_t(s)=F(s,t)$. Moreover, as $f_t$ coincides with $g_t$ on $T\cap S\times\{t\}$, we know that the bifurcation diagram of periodic orbits of the isotopy $g_t:T\cap \mathbb{R}^n\times\{t\}\to \mathbb{R}^n$ is a subdiagram of the bifurcation diagram of $f_t:S\to S$. The proof of the theorem for the case when $\Gamma$ is a tree is now complete.

\begin{figure}[h]
\centering
\begin{overpic}[width=0.35\textwidth]{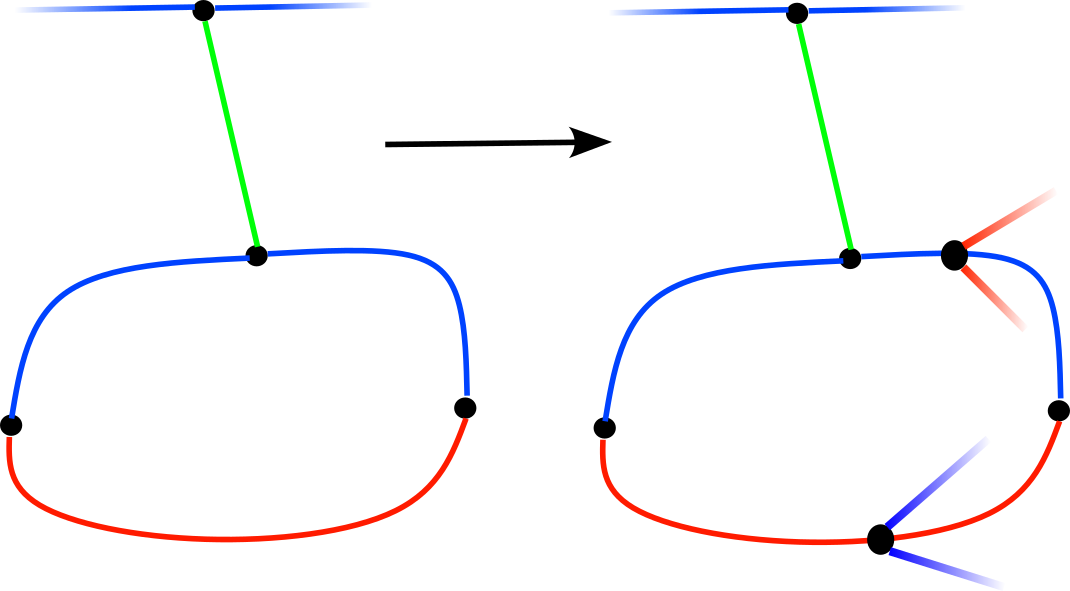}
\end{overpic}
\caption{\textit{Modifying the graph $\Gamma\in G_{1,2}$ which cannot be realized as a bifurcation diagram in $F_1(S^2)$, but it becomes an admissible bifurcation subdiagram in $F_2(S^3)$, hence lies on some subgraph of $G_2(S^3)$. Note that we do so by splitting two edges in two by adding a vertex corresponding to type $m$ bifurcations admissible in $S^3$, see Proposition \ref{rules3}.}}
\label{mod4}
\end{figure}

We now consider the case when $\Gamma$ is not a tree. We first note that whenever we pass through an $n$-junction, a type $m\geq3$ or a period-doubling bifurcation, the Lefschetz number is preserved. This implies the only difficulty can arise in a saddle-node bifurcation, when the diagram dictates two periodic orbits of different minimal period should collide and disappear (as is the case in Figure \ref{nonadmissible}). In this case, there are precisely two options:
\begin{itemize}
    \item \textbf{Case A} - one Orbit has Index $-1$, while the other has Mallet-Yorke Index $1$.
    \item \textbf{Case B} - both Orbits have Index $0$.
\end{itemize}

In both cases, the Lefschetz number associated with each periodic orbit is $-1$, and with the second, $1$. In this case, let us assume one periodic orbit has minimal period $j$ and another has minimal period $k$, and set $m=jkr$, s.t. $r$ is the minimal natural number for which $jkr\geq3$. We now modify $\Gamma$ into $\Gamma'$ by performing a type $m'$-bifurcation on both orbits, thus changing their periods and changing them from $k$ to $j$ (see the illustration in Figure \ref{mod4}) - which allows us to collide these orbits in a saddle-node bifurcation. Now, let us choose $d'\geq d$ as the minimal dimension which allows these bifurcations on $\Gamma'$ to take place - by the coloring of type $m$ bifurcations in dimensions $d\geq4$ given by Proposition \ref{rules4}, such a $d'$ exists. We now repeat the argument above with $\Gamma'$ instead of $\Gamma$, for some smooth, closed manifold $S'$ of dimension $d'$. 
\end{proof}

We now study which cycles cannot exist jointly in both $B_k(S)$ and $G_{k,d}$. We will not give a complete answer, only study a simple example showing the fine details of the difference between the two. To do so, we first remark that given any $d>1$, any $k>0$ and any even $n>0$, if $\Gamma_n$ is a cycle connecting $n$ vertices then we can color it s.t. $\Gamma_n\in G_{k,d}$. To see why it is so, note that $\Gamma_n$ has an even number of edges - therefore, setting each vertex as a saddle-node bifurcation vertex, by Propositions \ref{rules2}-\ref{rules4} we have precisely two options:
\begin{itemize}
    \item Color every edge with Mallet-Yorke Index $0$ (i.e., "green").
    \item Color half the edges with Mallet-Yorke Index $-1$ (i.e., "red"), and another half by Mallet-Yorke Index $1$ (i.e., "blue").
\end{itemize}
It is easy to see that for an even $n$, we also have $\Gamma_n\in B_k(S)$. On the other hand, when $n$ is odd the situation is very different. To illustrate, we prove:
\begin{lemma}
    \label{nocycles} Given an odd $n\in\mathbb{N}$ and any $d>2$, $G_{k,d}$ includes a cycle $\Gamma_n$ with an odd number of vertices - and when $d=2$, no such graph can exist in $G_{k,2}$. In contrast, for every closed smooth manifold $S$ of dimension $d>1$ no colored graph in $B_k(S)$ includes $\Gamma_n$ as a component.
\end{lemma}
\begin{proof}
We first consider the case when $d>2$. By Propositions \ref{rules3} and \ref{rules4}, the assumptions above yield that given any cycle with an odd number of vertices we can color it with Mallet-Yorke Index $0$ (see Figure \ref{fig5}) - we can do so as every vertex in $\Gamma_n$ has degree $2$, i.e., it is a saddle-node bifurcation. In other words, this proves $\Gamma_n\in G_{k,d}$. We now show there is no graph in $B_k(S)$ which allows $\Gamma_n$ as a component. To this end, let $f_t:S\to S$ be a $C^1$ one-parameter family in $F_k(S)$, $k>1$, $t\in[0,1]$. Due to the Snake Termination Principle in \cite{PY} we know that if $\Gamma$ is a bifurcation diagram for an orbit which is a cycle, we can smoothly embed it in $\mathbb{R}^2$ as a curve $\gamma:S^1\to \mathbb{R}^2$. By the Snake Termination Principle we also know $\gamma$ changes its orientation an even number of times, which implies it must have an even number of saddle-node bifurcations on it. This shows that $\Gamma_n$ cannot be a component of any graph in $B_k(S)$.

We now study the case when $d=2$. In this case, by Proposition \ref{rules2} we know that any saddle-node bifurcation can only take the form of a periodic orbit of Mallet-Yorke Index $1$ colliding with a periodic orbit of Mallet-Yorke Index $-1$. Or, put simply, in order to have a cyclic graph $\Gamma\in G_{k,d}$ we have to color $\Gamma$ with two colors, say, red and blue. By Proposition \ref{rules2} we know that every vertex on $\Gamma$ has to be the meeting point of a blue and a red edge, being a saddle-node bifurcation. This implies $\Gamma$ must have an even number of red and blue edges, and as it is a cycle, it must have an even number of vertices. This proves that any cyclic graph $\Gamma\in G_{k,d}$ has an even number of vertices and the assertion follows.
\end{proof}

Before concluding this appendix, we remark that characterizing the complete relationship between $B_k(S)$ and $G_{k,d}$ probably cannot be solved using graph-theoretical tools alone. Specifically, in order to do so the topology of $S$ has to be taken into account. The reason we are led to think so is due to Nielsen Theory. Briefly, recall that given a homeomorphism $f:S\to S$ (where $S$ is some manifold, not necessarily closed), Nielsen Theory studies which collections of periodic orbits persist under isotopies, or more generally, homotopies (for a survey of how this theory can be applied to surface dynamics, see \cite{Bo} and \cite{FAS}). Such persistent periodic orbits are called \textbf{unremovable} - and from the theory of two-dimensional dynamical systems we know the unremovability of a periodic orbit w.r.t. a surface isotopy is strongly connected to the topology of $S$ (see \cite{BeH} and \cite{Bo} for the complete details). This teaches us that in order to solve this problem, given a $C^1$ one-parameter family $f_t:S\to S$, $t\in[0,1]$ in $F_k(S)$, we have to find out what unremovable orbits exist for $f_0$ and then study how they can (and cannot) bifurcate as we vary $t$ towards $[0,1]$ (in addition, one would also have to somehow correlate these results with the Mallet-Yorke Index Theory).

\section{Data availability statement}
No data was used for this paper.
\section{Competing interests statement}
The authors state that there is no conflict of interest, and that they did not receive support from any organization for the submitted work. Moreover, the authors have no financial or proprietary interests in any material discussed in this article.
\section{Funding statement}
The authors did not receive support or funding from any organization for the submitted work.
\section{Author contribution statement}
 Both authors made an equal contribution to this study, and both comply with the COPE regulations.
 
\newpage
\printbibliography

\end{document}